\documentclass[12pt]{amsart}
\usepackage{fullpage}
\usepackage{multicol}
\usepackage{tikz}
\usepackage{hyperref}
\usetikzlibrary{calc}
\usetikzlibrary{decorations.markings}
\tikzset{->-/.style={decoration={
  markings,
  mark=at position 0.55 with {\arrow{>}}},postaction={decorate}}}

\newcommand{\Pio}{\mathring{\Pi}}
\def\tX{{\tilde X}}
\def\tGr{\widetilde \Gr}
\def\sort{{\rm sort}}
\def\R{{\mathbb R}}
\def\M{{\mathcal{M}}}
\def\tM{\tilde M}
\def\Gr{{\rm Gr}}
\def\C{{\mathbb C}}
\def\wt{{\rm wt}}
\def\Z{{\mathbb Z}}
\def\D{{\mathcal{D}}}
\def\A{{\mathcal{A}}}
\def\W{{\mathcal{W}}}
\newcommand{\defn}[1]{{\bf #1}}

\newtheorem{proposition}{Proposition}
\newtheorem{lemma}[proposition]{Lemma}
\newtheorem{theorem}[proposition]{Theorem}
\newtheorem{corollary}[proposition]{Corollary}
\newtheorem{remark}[proposition]{Remark}
\newtheorem{example}[proposition]{Example}
\newtheorem{conj}[proposition]{Conjecture}
\numberwithin{proposition}{section}
\begin{document}
\title{Dimers, webs, and positroids}
\author{Thomas Lam}\address{Department of Mathematics, University of Michigan,
2074 East Hall, 530 Church Street, Ann Arbor, MI 48109-1043, USA}
\email{tfylam@umich.edu}\thanks{T.L. was supported by NSF grant DMS-1160726.}
\begin{abstract}
We study the dimer model for a planar bipartite graph $N$ embedded in a disk, with boundary vertices on the boundary of the disk.  Counting dimer configurations with specified boundary conditions gives a point in the totally nonnegative Grassmannian.  Considering pairing probabilities for the double-dimer model gives rise to Grassmann analogues of Rhoades and Skandera's Temperley-Lieb immanants.  The same problem for the (probably novel) triple-dimer model gives rise to the combinatorics of Kuperberg's webs and Grassmann analogues of Pylyavskyy's web immanants.  This draws a connection between the square move of plabic graphs (or urban renewal of planar bipartite graphs), and Kuperberg's square reduction of webs.  Our results also suggest that canonical-like bases might be applied to the dimer model.

We furthermore show that these functions on the Grassmannian are compatible with restriction to positroid varieties.  Namely, our construction gives bases for the degree two and degree three components of the homogeneous coordinate ring of a positroid variety that are compatible with the cyclic group action.
\end{abstract}
\maketitle

\section{Introduction}

Let $N$ be a (weighted) planar bipartite graph embedded into the disk with $n$ boundary vertices labeled $1,2,\ldots,n$ clockwise along the boundary of the disk.  We study some algebraic aspects of the dimer configurations of $N$.

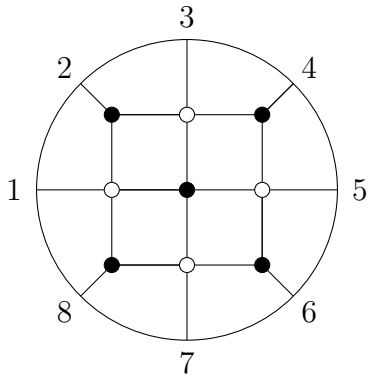
\begin{center}
\begin{tikzpicture}
\node at (180:2.3) {$1$};
\node at (135:2.3) {$2$};
\node at (90:2.3) {$3$};
\node at (45:2.3) {$4$};
\node at (0:2.3) {$5$};
\node at (-45:2.3) {$6$};
\node at (270:2.3) {$7$};
\node at (225:2.3) {$8$};
\coordinate (a4) at (45:2);
\coordinate (a3) at (90:2);
\coordinate (a2) at (135:2);
\coordinate (a1) at (180:2);
\coordinate (a8) at (225:2);
\coordinate (a7) at (270:2);
\coordinate (a6) at (315:2);
\coordinate (a5) at (0:2);
\coordinate (x11) at (-1,1);
\coordinate (x12) at (0,1);
\coordinate (x13) at (1,1);
\coordinate (x21) at (-1,0);
\coordinate (x22) at (0,0);
\coordinate (x23) at (1,0);
\coordinate (x31) at (-1,-1);
\coordinate (x32) at (0,-1);
\coordinate (x33) at (1,-1);

\draw (a1) -- (x21);
\draw (a2) -- (x11);
\draw (a3) -- (x12);
\draw (a4) -- (x13);
\draw (a5) -- (x23);
\draw (a6) -- (x33);
\draw (a7) -- (x32);
\draw (a8) -- (x31);
\draw (x11) -- (x13);
\draw (x21) -- (x23);
\draw (x31) -- (x33);
\draw (x11) -- (x31);
\draw (x12) -- (x32);
\draw (x13) -- (x33);
\draw (0,0) circle (2cm);
\draw(x11) -- (x12);
\draw (x13) -- (a4);
\draw (x21) -- (x22);
\draw (x31) -- (x32);
\draw (x23) -- (x33);

\filldraw[black] (x11) circle (0.1cm);
\filldraw[black] (x13) circle (0.1cm);
\filldraw[black] (x22) circle (0.1cm);
\filldraw[black] (x31) circle (0.1cm);
\filldraw[black] (x33) circle (0.1cm);

\filldraw[white] (x12) circle (0.1cm);
\draw (x12) circle (0.1cm);
\filldraw[white] (x21) circle (0.1cm);
\draw (x21) circle (0.1cm);
\filldraw[white] (x23) circle (0.1cm);
\draw (x23) circle (0.1cm);
\filldraw[white] (x32) circle (0.1cm);
\draw (x32) circle (0.1cm);
\end{tikzpicture}
\end{center}

\subsection{Dimers and the totally nonnegative Grassmannian}

A \defn{dimer configuration}, or almost perfect matching $\Pi$ of $N$ is a collection of edges in $N$ that uses each interior vertex exactly once, and some subset of the boundary vertices.  The data of the subset of boundary vertices that are used in $\Pi$ gives a boundary subset $I(\Pi) \subset [n]$.  We define a generating function
$$
\Delta_I(N) = \sum_{\Pi: I(\Pi) = I} \wt(\Pi)
$$
where the weight $\wt(\Pi)$ is the product of the weights of edges used in $\Pi$.  These \defn{boundary measurements} satisfy Pl\"ucker relations and gives (see \cite{Kuo, Tal, PSW} and Theorem \ref{thm:matchingplucker}) a point $\tM(N)$ in the affine cone $\tGr(k,n)$ over the Grassmannian of $k$-planes in $n$-space (the value of $k$ depends on $N$):
$$
N \rightsquigarrow \tM(N) = (\Delta_I)_{I \in \binom{[n]}{k}} \in \tGr(k,n).
$$
Indeed, the image $M(N) \in \Gr(k,n)$ of $\tM(N)$ lies in Postnikov's totally nonnegative Grassmannian $\Gr(k,n)_{\geq 0}$ \cite{Pos}, which is defined to be the set of points in the Grassmannian where all Pl\"ucker coordinates take nonnegative values.

\subsection{Double-dimers and Temperley-Lieb immanants}

A \defn{double-dimer configuration} in $N$ is an ordered pair $(\Pi,\Pi')$ of two dimer configurations in $N$.  Overlaying the two dimer configurations gives a collection of doubled edges, cycles of even length, and paths between boundary vertices (see the picture in Section \ref{ssec:doubledimer}), which we call a \defn{Temperley-Lieb subgraph}.  The paths between boundary vertices gives a non-crossing pairing of some subset of the boundary vertices, studied for example by Kenyon and Wilson \cite{KW}.  Let $\A_n = \{(\tau,T)\}$ (notation to be explained in Section \ref{sec:doubledimer}) denote the set of partial non-crossing pairings $(\tau,T)$ on $n$-vertices.  If in addition we fix the boundary subsets of $\Pi$ and $\Pi'$, this analysis gives the identity (Theorem \ref{thm:TL})
\begin{equation}\label{eq:TL}
\Delta_I(N)\Delta_J(N) = \sum_{(\tau,T)} F_{\tau,T}(N)
\end{equation}
where $F_{\tau,T}(N)$ is a \defn{Temperley-Lieb immanant} defined as the weight generating function of Temperley-Lieb subgraphs in $N$ with specified partial non-crossing pairing $(\tau,T)$, and the summation is over certain non-crossing pairings $(\tau,T) \in \A_n$ that are compatible with $(I,J)$.  We show (Proposition \ref{prop:tau}) that $F_{\tau,T}$ are functions on the cone $\tGr(k,n)$ over the Grassmannian: that is, $F_{\tau,T}(N)$ only depends on $\tM(N)$.  The functions $F_{\tau,T}$ are Grassmann analogues of the Temperley-Lieb immanants of Rhoades and Skandera \cite{RS}.

Equation \eqref{eq:TL} leads to some inequalities between minors on $\Gr(k,n)_{\geq 0}$.  For example, we have (Proposition \ref{prop:inequalities})
$$
\Delta_{\sort_1(I,J)}(X) \Delta_{\sort_2(I,J)}(X) \geq
\Delta_I(X) \Delta_J(X)
$$
for $X \in \Gr(k,n)_{\geq 0}$, where $\sort_1,\sort_2$ are defined in Section \ref{sec:logconcave}.  This inequality was also independently discovered by Farber and Postnikov \cite{FaPo}.  Indeed, there is an analogy between these inequalities and the Schur function inequalities of Lam, Postnikov, and Pylyavskyy \cite{LPP}.

\subsection{Triple-dimers and web immanants}

A \defn{triple-dimer configuration} is an ordered triple $(\Pi_1,\Pi_2,\Pi_3)$ of dimer configurations.  Overlaying these dimer configurations on top of each other, we obtain a \defn{weblike} subgraph $G \subset N$ (to be defined in the text) consisting of some tripled edges, some even length cycles that alternate between single and doubled edges, and some components illustrated below (thick edges are present in two out of the three dimer configurations):

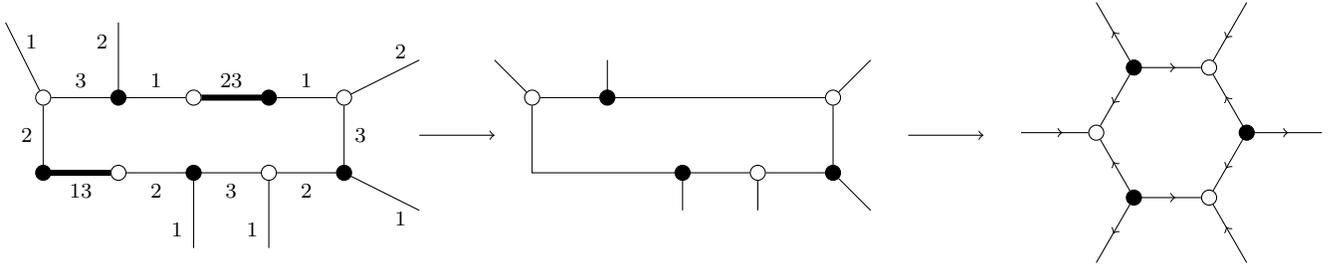
\begin{center}
\begin{tikzpicture}[font = \tiny]
\coordinate (b) at (0,0);
\coordinate (w) at (1,0);
\coordinate (b1) at (2,0);
\coordinate (w1) at (3,0);
\coordinate (w2) at (-1,0);
\coordinate (b2) at (-1,-1);
\coordinate (w3) at (0,-1);
\coordinate (b3) at (1,-1);
\coordinate (w4) at (2,-1);
\coordinate (b4) at (3,-1);

\draw (b) -- node[above] {$1$} (w) -- node[above] {$23$} (b1) -- node[above] {$1$}(w1) -- node[right] {$3$} (b4) -- node[below] {$2$} (w4) --node[below] {$3$} (b3) -- node[below] {$2$}(w3) -- node[below] {$13$}(b2) -- node[left] {$2$} (w2) -- node[above] {$3$}(b);

\draw[line width=0.8 mm] (w) -- (b1);
\draw[line width=0.8 mm] (b2) -- (w3);
\draw (b) -- node[near end, left] {$2$} (0,1);
\draw (w2) -- node[near end, right] {$1$}(-1.5,1);
\draw (b3) -- node[near end, left] {$1$}(1,-2);
\draw (w4) -- node[near end, left] {$1$}(2,-2);
\draw (b4) --node[near end, below] {$1$}(4,-1.5);
\draw (w1) -- node[near end, above] {$2$}(4,0.5);

\filldraw [white] (w) circle (0.1cm);
\filldraw [black] (b) circle (0.1cm);
\draw (w) circle (0.1cm);
\filldraw [white] (w1) circle (0.1cm);
\filldraw [black] (b1) circle (0.1cm);
\draw (w1) circle (0.1cm);
\filldraw [white] (w2) circle (0.1cm);
\filldraw [black] (b2) circle (0.1cm);
\draw (w2) circle (0.1cm);
\filldraw [white] (w3) circle (0.1cm);
\filldraw [black] (b3) circle (0.1cm);
\draw (w3) circle (0.1cm);
\filldraw [white] (w4) circle (0.1cm);
\filldraw [black] (b4) circle (0.1cm);
\draw (w4) circle (0.1cm);

\draw[->] (4,-0.5) -- (5,-0.5);

\begin{scope}[{shift={(6.5,0)}}]
\coordinate (b) at (0,0);
\coordinate (w) at (1,0);
\coordinate (b1) at (2,0);
\coordinate (w1) at (3,0);
\coordinate (w2) at (-1,0);
\coordinate (b2) at (-1,-1);
\coordinate (w3) at (0,-1);
\coordinate (b3) at (1,-1);
\coordinate (w4) at (2,-1);
\coordinate (b4) at (3,-1);

\draw (b) -- (w) -- (b1) -- (w1) -- (b4) -- (w4) -- (b3) -- (w3) -- (b2) -- (w2) -- (b);

\draw (b) -- (0,0.5);
\draw (w2) -- (-1.5,0.5);
\draw (b3) -- (1,-1.5);
\draw (w4) -- (2,-1.5);
\draw (b4) -- (3.5,-1.5);
\draw (w1) -- (3.5,0.5);

\filldraw [black] (b) circle (0.1cm);
\filldraw [white] (w1) circle (0.1cm);
\draw (w1) circle (0.1cm);
\filldraw [white] (w2) circle (0.1cm);
\draw (w2) circle (0.1cm);
\filldraw [black] (b3) circle (0.1cm);
\filldraw [white] (w4) circle (0.1cm);
\filldraw [black] (b4) circle (0.1cm);
\draw (w4) circle (0.1cm);
\end{scope}
\draw[->] (10.5,-0.5) -- (11.5,-0.5);
\begin{scope}[{shift={(13.5,0.4)}}]
\coordinate (b) at (0,0);
\coordinate (w) at (1,0);
\coordinate (b1) at (1.5,-0.866);
\coordinate (w1) at (1,-1.73);

\coordinate (b2) at (0,-1.73);
\coordinate (w2) at (-0.5,-0.866);

\draw[->-] (b) -- (w);
\draw[->-] (b1) -- (w);
\draw[->-] (b2) -- (w2); 
\draw[->-] (b) -- (w2);
\draw[->-] (b1) -- (w1);
\draw[->-] (b2) -- (w1);

\begin{scope}[{shift={(b)}}]
\draw[->-] (0,0) -- (-0.5,0.866);
\end{scope}
\begin{scope}[{shift={(w)}}]
\draw[->-] (0.5,0.866) -- (0,0);
\end{scope}
\begin{scope}[{shift={(b1)}}]
\draw[->-] (0,0) -- (1,0);
\end{scope}
\begin{scope}[{shift={(w1)}}]
\draw[->-] (0.5,-0.866) -- (0,0);
\end{scope}
\begin{scope}[{shift={(b2)}}]
\draw[->-] (0,0) -- (-0.5,-0.866);
\end{scope}
\begin{scope}[{shift={(w2)}}]
\draw[->-] (-1,0) -- (0,0);
\end{scope}

\filldraw [white] (w) circle (0.1cm);
\filldraw [black] (b) circle (0.1cm);
\draw (w) circle (0.1cm);
\filldraw [white] (w1) circle (0.1cm);
\filldraw [black] (b1) circle (0.1cm);
\draw (w1) circle (0.1cm);
\filldraw [white] (w2) circle (0.1cm);
\filldraw [black] (b2) circle (0.1cm);
\draw (w2) circle (0.1cm);
\end{scope}
\end{tikzpicture}
\end{center}

Informally, these components consist of trivalent vertices joined together by paths that alternate between single and doubled edges.  Such a weblike graph gives rise to a web $W$ (shown on the right) in the sense of Kuperberg \cite{Kup}.  Kuperberg's webs have directed edges, and our bipartite webs should be interpreted with all edges directed towards white interior vertices.  Kuperberg gave a reduction algorithm for such graphs, reducing any web to a linear combination of 
\defn{non-elliptic webs}.

The set of non-elliptic webs on $n$ boundary vertices, denoted $\D_n$, should be thought of as the set of possible connections in a triple dimer configuration.  For each $D \in \D_n$ we define a generating function $F_D(N)$, counting weblike subgraphs $G \subset N$, called a \defn{web immanant} and we show that $F_D(N)$ only depends on $\tM(N)$.  In particular, if $N$ and $N'$ are related by certain moves, such as the square move (also called urban renewal), then $F_D(N) = \alpha_{N,N'} F_D(N')$ for a constant $\alpha_{N,N'}$ not depending on $D$.  We also obtain (Theorem \ref{thm:triple}) an identity
$$
\Delta_I \Delta_J \Delta_K = \sum_D a(I,J,K,D) F_D
$$
where $a(I,J,K,D)$ counts the number of ways to ``consistently label" $D$ with $(I,J,K)$.  This is a Grassmann analogue of a result of Pylyavskyy \cite{Pyl}.

\subsection{Boundary, pairing, and web ensembles in planar bipartite graphs}
Given a planar bipartite graph $N$, we may define 
$$
\M(N): = \{I(\Pi)\} = \left\{I \in \binom{[n]}{k} \mid \Delta_I(N) > 0\right\}
$$
to be the collection of boundary subsets $I = I(\Pi)$ that occur in dimer configurations $\Pi$ in $N$.  Similarly, one defines
$$\A(N) := \{(\tau,T) \in \A_n \mid F_{\tau,T}(N) > 0\}$$ 
to be the collection of partial non-crossing pairings $(\tau,T)$ that occur in double-dimers in $N$, and
$$
\D(N)  := \{D \in \D_n \mid F_{D}(N) > 0\}
$$ 
to be the collection of web connections $D$ that occur in triple-dimers in $N$.  It is not obvious (but follows from our results) that knowing $\M(N)$ determines both $\A(N)$ and $\D(N)$.  We propose to call $\M(N)$, $\A(N)$ and $\D(N)$ the \defn{boundary ensemble}, \defn{pairing ensemble}, and \defn{web ensemble} of $N$ respectively.

\subsection{Positroids and bases of homogeneous coordinate rings of positroid varieties}

If $X \in \Gr(k,n)$ the matroid of $X$ is the collection
$$
\M_X := \left\{I \in \binom{[n]}{k} \mid \Delta_I(X) \neq 0\right\}
$$
of $k$-element subsets labeling non-vanishing Pl\"ucker coordinates.  A matroid $\M$ is a \defn{positroid} if $\M = \M_X$ for some $X \in \Gr(k,n)_{\geq 0}$.  Thus $\M(N) = \M_{M(N)}$ is always a positroid, and it follows from Postnikov's work \cite{Pos} that every positroid occurs in this way.

The positroid stratification \cite{KLS, Pos} is the stratification $\Gr(k,n) = \bigcup_\M \Pio_\M$ obtained by intersecting $n$ cyclically rotated Schubert stratifications.  Each such stratum is labeled by a positroid $\M$.  We denote the corresponding closed positroid variety by $\Pi_\M$ and the open stratum by $\Pio_\M$.  For any $X \in (\Pi_\M)_{>0} = \Pio_\M \cap \Gr(k,n)_{\geq 0}$, we have $\M_X = \M$ -- so all totally nonnegative points in an open positroid stratum have the same matroid.  Picking $X \in (\Pi_\M)_{>0}$ one defines
$$
\A(\M) := \{(\tau, T) \in \A_n \mid F_{\tau,T}(X) > 0\}
$$
and 
$$
\D(\M) := \{D \in \D_n \mid F_{D}(X) > 0\}.
$$
We show that $\A(\M)$ and $\D(\M)$ do not depend on the choice of $X$, but only $\M$.  In particular $\A(N) = \A(\M(N))$ and $\D(N) = \D(\M(N))$.  It would be interesting to give a direct description of $\A(\M)$ and $\D(\M)$ similar to Oh's description \cite{Oh} of $\M$ as an intersection of cyclically rotated Schubert matroids.

Let $\C[\Pi_\M]$ denote the homogeneous coordinate ring of a positroid variety.  We prove the following statements: 
\begin{enumerate}
\item For each positroid $\M$, the set
$$
\{ F_{\tau,T} \mid (\tau,T) \in \A(\M) \}
$$
forms a basis of the degree 2 part of $\C[\Pi_\M]$ (Theorem \ref{thm:TLbasis}).  
\item
The set
$$
\{ F_{D} \mid D \in \D(\M) \}
$$
forms a basis of the degree 3 part of $\C[\Pi_\M]$ (Theorem \ref{thm:web}). 
\end{enumerate}
Thus we have a combinatorially defined, cyclically invariant basis for these parts of the homogeneous coordinate rings.  These bases are likely related to (but not identical to, see \cite{KK}) Lusztig's dual canonical basis.  We remark that Launois and Lenagan \cite{LL} have studied the cyclic action on the quantized coordinate ring of the Grassmannian.

There is also a relation to cluster structures on Grassmannians and positroid varieties that for simplicity I have chosen to omit discussing in this work.  We note that Fomin and Pylyavskyy \cite{FoPy} have constructed, using generalizations of Kuperberg's webs, bases of certain rings of invariants, that include Grassmannians of $3$-planes as special cases.  Marsh and Scott \cite{MaSc} have investigated twists of Grassmannians in terms of dimer configurations.  Recently, cluster structures related to the coordinate rings of positroid varieties have also been studied by Leclerc \cite{Lec} and Muller and Speyer \cite{MuSp}. 

We hope to return to the connection with canonical and semicanonical bases, and cluster structures in the future.

{\bf Acknowledgements.} We thank Milen Yakimov for pointing us to \cite{LL}.

\section{The dimer model and the totally nonnegative Grassmannian}
\subsection{TNN Grassmannian}
In this section, we fix integers $k,n$ and consider the real Grassmannian $\Gr(k,n)$ of (linear) $k$-planes in $\R^n$.  Recall that each $X \in \Gr(k,n)$ has Pl\"ucker coordinates $\Delta_I(X)$ labeled by $k$-element subsets $I \subset [n]$, defined up to a single common scalar.  It will be convenient for us to talk about the Pl\"ucker coordinates $\Delta_I$ as genuine functions.  We will thus often work with the affine cone $\tGr(k,n)$ over the Grassmannian.  A point in $\tX \in \tGr(k,n)$ is given by a collection of Pl\"ucker coordinates $\Delta_I(\tX)$, satisfying the Pl\"ucker relations \cite{Ful} (without the equivalence relation where we scale all coordinates by a common scalar).

Suppose $\tX, \tX' \in \tGr(k,n)$ represent the same point in $\Gr(k,n)$.  Then there exists a non-zero scalar $a \in \R$ such that $\Delta_I(\tX) = a \Delta_I(\tX')$ for all $I \in \binom{[n]}{k}$.  As a shorthand we then write $\tX = a \tX'$.

The TNN (totally nonnegative) Grassmannian $\Gr(k,n)_{\geq 0}$ is the subset of $\Gr(k,n)$ consisting of points $X$ represented by nonnegative Pl\"ucker coordinates $\{\Delta_I(X) \mid I \in \binom{[n]}{k}\}$.  Similarly one can define the TNN part $\tGr(k,n)_{\geq 0}$ of the cone over the Grassmannian.

The cyclic group acts on $\Gr(k,n)_{\geq 0}$ (and on $\tGr(k,n)_{\geq 0}$) with generator $\chi$ acting by the map
$$
\chi: \left(v_1,v_2,\ldots,v_n\right) \mapsto \left(v_2,\ldots,v_n, (-1)^{k-1}v_1\right)
$$
where $v_i$ are columns of some $k\times n$ matrix representing $X$.

\subsection{Dimer model for a bipartite graph with boundary vertices}
Let $N$ be a weighted bipartite network embedded in the disk with $n$ boundary vertices, labeled $1,2,\ldots,n$ in clockwise order.  Each vertex (including boundary vertices) is colored either black or white, and all edges join black vertices to white vertices.  We let $d$ be the number of interior white vertices minus the the number of interior white vertices.  Furthermore we let $d' \in [n]$ be the number of white boundary vertices.  Finally, we assume that all boundary vertices have degree one, and that edges cannot join boundary vertices to boundary vertices.  We shall also use the standard convention that in our diagrams unlabeled edges have weight $1$.

Since the graph is bipartite, the condition that boundary vertices have degree one ensures that the coloring of the boundary vertices is determined by the interior part of the graph.  So we will usually omit the color of boundary vertices from pictures. 

A \defn{dimer configuration} or almost perfect matching $\Pi$ is a subset of edges of $N$ such that 
\begin{enumerate}
\item
each interior vertex is used exactly once
\item
boundary vertices may or may not be used.
\end{enumerate}
The boundary subset $I(\Pi) \subset \{1,2,\ldots,n\}$ is the set of black boundary vertices that are used by $\Pi$ union the set of white boundary vertices that are not used.  By our assumptions we have $|I(\Pi)| =  k:=d' + d$.

Define the \defn{boundary measurement} $\Delta_I(N)$ as follows.  For $I \subset [n]$ a $k$-element subset,
$$
\Delta_I(N) = \sum_{\Pi : I(\Pi) = I} \wt(\Pi)
$$
where $\wt(\Pi)$ is the product of the weight of the edges in $\Pi$.  The first part of the following result is essentially due to Kuo \cite{Kuo} and we will prove it using the language of Temperley-Lieb immanants in Section \ref{sec:doubledimer}.  The second part of the  theorem is due to Postnikov \cite{Pos} who counted paths instead of matchings; see also \cite{Lamnotes} for a proof in the spirit of the current work.  The relation between Postnikov's theory and the dimer model was suggested by the works of Talaska \cite{Tal} and Postnikov, Speyer and Williams \cite{PSW}.

\begin{theorem}\label{thm:matchingplucker}
Suppose $N$ has nonnegative real weights.  Then the coordinates $(\Delta_I(N))_{I \in \binom{[n]}{k}})$ defines a point $\tM(N)$ in the cone over the Grassmannian $\tGr(k,n)_{\geq 0}$.  Furthermore, every $X \in \tGr(k,n)_{\geq 0}$ is realizable $X = \tM(N)$ by a planar bipartite graph.
\end{theorem}

We let $M(N)$ denote the equivalence class of $\tM(N)$ in $\Gr(k,n)$.  We will often implicitly assume that $N$ does have dimer configurations, so that $M(N)$ is well-defined.

\subsection{Gauge equivalences and local moves}\label{sec:moves}
We now discuss operations on $N$ that preserve $M(N)$.

Let $N$ be a planar bipartite graph.  If $e_1,e_2,\ldots,e_d$ are incident to an interior vertex $v$, we can multiply all of their edge weights by the same constant $c \in \R_{> 0}$ to get a new graph $N'$, and we have $M(N') = M(N)$.  This is called a gauge equivalence.

We also have the following local moves, replacing a small local part of $N$ by another specific graph to obtain $N'$:
\begin{enumerate}
\item[(M1)]
Spider move \cite{GK}, square move \cite{Pos}, or urban renewal \cite{Propp,Ciu}: assuming the leaf edges of the spider have been gauge fixed to 1, the transformation is
$$
a'=\frac{a}{ac+bd} \qquad b'=\frac{b}{ac+bd} \qquad c'=\frac{c}{ac+bd} \qquad d'=\frac{d}{ac+bd}
$$
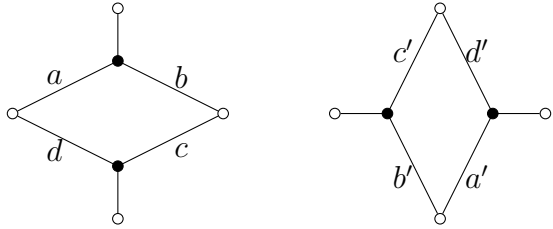
\begin{center}
\begin{tikzpicture}[scale=0.7]
\draw (-2,0) -- (0,1)--(2,0)--(0,-1)--  (-2,0);
\draw (0,1) -- (0,2);
\draw (0,-1) -- (0,-2);
\node at (-1.2,0.7) {$a$};
\node at (-1.2,-0.7) {$d$};
\node at (1.2,0.7) {$b$};
\node at (1.2,-0.7) {$c$};

\filldraw[black] (0,1) circle (0.1cm);
\filldraw[black] (0,-1) circle (0.1cm);
\filldraw[white] (-2,0) circle (0.1cm);
\draw (-2,0) circle (0.1cm);
\filldraw[white] (0,2) circle (0.1cm);
\draw (0,2) circle (0.1cm);
\filldraw[white] (2,0) circle (0.1cm);
\draw (2,0) circle (0.1cm);
\filldraw[white] (0,-2) circle (0.1cm);
\draw (0,-2) circle (0.1cm);
\end{tikzpicture}
\hspace{30pt}
\begin{tikzpicture}[scale=0.7]
\draw (0,-2) -- (1,0)-- (0,2)-- (-1,0)-- (0,-2);
\draw (1,0) -- (2,0);
\draw (-1,0) -- (-2,0);
\node at (0.7,-1.2) {$a'$};
\node at (-0.7,-1.2) {$b'$};
\node at (0.7,1.2) {$d'$};
\node at (-0.7,1.2) {$c'$};

\filldraw[black] (1,0) circle (0.1cm);
\filldraw[black] (-1,0) circle (0.1cm);
\filldraw[white] (-2,0) circle (0.1cm);
\draw (-2,0) circle (0.1cm);
\filldraw[white] (0,2) circle (0.1cm);
\draw (0,2) circle (0.1cm);
\filldraw[white] (2,0) circle (0.1cm);
\draw (2,0) circle (0.1cm);
\filldraw[white] (0,-2) circle (0.1cm);
\draw (0,-2) circle (0.1cm);
\end{tikzpicture}
\end{center}
\item[(M2)]
Valent two vertex removal.  If $v$ has degree two, we can gauge fix both incident edges $(v,u)$ and $(v,u')$ to have weight 1, then contract both edges (that is, we remove both edges, and identify $u$ with $u'$).  Note that if $v$ is a valent two-vertex adjacent to boundary vertex $b$, with edges $(v,b)$ and $(v,u)$, then removing $v$ produces an edge $(b,u)$, and the color of $b$ flips.

\item[(R1)]
Multiple edges with same endpoints is the same as one edge with sum of weights.
\item[(R2)]
Leaf removal.  Suppose $v$ is leaf, and $(v,u)$ the unique edge incident to it.  Then we can remove both $v$ and $u$, and all edges incident to $u$.  However, if there is a boundary edge $(b,u)$ where $b$ is a boundary vertex, then that edge is replaced by a boundary edge $(b,w)$ where $w$ is a new vertex with the same color as $v$.
\item[(R3)]
Dipoles (two degree one vertices joined by an edge) can be removed.
\end{enumerate}

The following result is a case-by-case check.
\begin{proposition}
Each of these relations preserves $M(N)$.
\end{proposition}

The following result is due to Postnikov \cite{Pos} in the more general setting of plabic graphs.

\begin{theorem}\label{thm:Pos}
Suppose $N$ and $N'$ are planar bipartite graphs with $M(N) = M(N')$.  Then $N$ and $N'$ are related by local moves and gauge equivalences.  
\end{theorem}

\subsection{Positroid stratification}
Let $X \in \Gr(k,n)$.  The matroid $\M_X$ of $X$ is the collection 
$$
\M_X := \left\{ I \in \binom{[n]}{k} \mid \Delta_I(X) \neq 0\right\}
$$ 
of $k$-element subsets of $[n]$ labeling non-vanishing Pl\"ucker coordinates of $X$.  If $X \in \Gr(k,n)_{\geq 0}$ then $\M$ is called a \defn{positroid}.  Unlike matroids in general, positroids have been completely classified and characterized \cite{Pos}.  Oh \cite{Oh} shows that positroids are exactly the intersections of cyclically rotated Schubert matroids.  Lam and Postnikov \cite{LP} show that positroids are exactly the matroids that are closed under sorting (see Section \ref{sec:logconcave}).

We have a stratification
$$
\Gr(k,n) = \bigcup_\M \Pio_\M
$$ 
of the Grassmannian by \defn{open positroid varieties}, labeled by positroids $\M$.  The strata $\Pio_\M$ are defined as the intersections of cyclically rotated Schubert cells (see \cite{KLS}).  The closure $\Pi_\M$ of $\Pio_\M$ is an irreducible subvariety of the Grassmannian called a \defn{(closed) positroid variety}.  Postnikov \cite{Pos} showed  

\begin{theorem}\label{thm:cell}\
\begin{enumerate}
\item
The intersection $(\Pi_\M)_{>0} = \Gr(k,n)_{\geq 0} \cap \Pio_\M$ is homeomorphic to $\R_{>0}^d$, where $d = \dim(\Pi_\M)$.  
\item
For each positroid $\M$, there exists a planar bipartite graph $N_\M = N_\M(t_1,t_2,\ldots,t_d)$, where $d$ of the edges have weights given by parameters $t_1,\ldots,t_d$ and all other weights are 1, such that
$$
(t_1,t_2,\ldots, t_d) \mapsto M(N_\M(t_1,t_2,\ldots,t_d))
$$
is a parametrization of $(\Pi_\M)_{>0}$ as $(t_1,t_2,\ldots, t_d)$ vary over $\R_{>0}^d$.
\end{enumerate}
\end{theorem}

In particular, positroids can be characterized completely in terms of planar bipartite graphs.  Namely, $\M$ is a positroid if and only if it is the matroid $\M_X$ of a point $X = M(N)$ where $N$ is a planar bipartite graph.

It follows from Theorem \ref{thm:cell} that $(\Pi_\M)_{>0}$ is Zariski-dense in $\Pi_\M$.  We shall construct elements of the homogeneous coordinate ring $\C[\Pi_\M]$ using the combinatorics of planar bipartite graphs.

\subsection{Bridge and lollipop recursion}
We will require two additional operations on planar bipartite graphs that do not preserve $M(N)$.  The first operation is \defn{adding a bridge at $i$}, black at $i$ and white at $i+1$.  It modifies a bipartite graph near the boundary vertices $i$ and $i+1$:
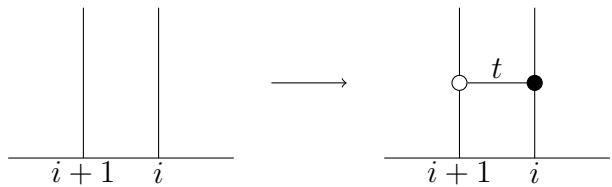
\begin{center}
\begin{tikzpicture}
\draw (1,0) -- (4,0);
\draw (2,0) -- (2,2);
\draw (3,0) -- (3,2) ;
\node at (2,-0.2) {$i+1$};
\node at (3,-0.2) {$i$};

\draw[->] (4.5,1) -- (5.5,1);
\begin{scope}[{shift={(5,0)}}]
\draw (1,0) -- (4,0);
\draw (2,0) -- (2,2);
\draw (3,0) -- (3,2) ;
\draw (2,1) -- (3,1);
\filldraw [white] (2,1) circle (0.1cm);
\filldraw [black] (3,1) circle (0.1cm);
\draw (2,1) circle (0.1cm);
\node at (2,-0.2) {$i+1$};
\node at (3,-0.2) {$i$};
\node at (2.5,1.2) {$t$};
\end{scope}
\end{tikzpicture}
\end{center}

The bridge edge is the edge labeled $t$ in the above picture.  Note that in general this modification might create a graph that is not bipartite -- for example, if in the original graph $i$ is connected to a black vertex.  However, by adding valent two vertices using local move (M2), we can always assume we obtain a bipartite graph.

The second operation is \defn{adding a lollipop} which can be either white or black.  This inserts a new boundary vertex connected to an interior leaf.  The new boundary vertices are then relabeled:
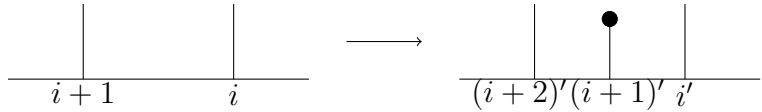
\begin{center}
\begin{tikzpicture}
\draw (1,0) -- (5,0);
\draw (2,0) -- (2,1);
\draw (4,0) -- (4,1) ;
\node at (2,-0.2) {$i+1$};
\node at (4,-0.2) {$i$};

\draw[->] (5.5,0.5) -- (6.5,0.5);
\begin{scope}[{shift={(6,0)}}]
\draw (1,0) -- (5,0);
\draw (2,0) -- (2,1);
\draw (4,0) -- (4,1) ;
\draw (3,0) -- (3,0.8);
\filldraw [black] (3,0.8) circle (0.1cm);
\node at (1.8,-0.2) {$(i+2)'$};
\node at (3.1,-0.2) {$(i+1)' $};
\node at (4,-0.2) {$i'$};
\end{scope}
\end{tikzpicture}
\end{center}

In the following we will use $N_\M$ to denote any parameterized planar bipartite graph satisfying Theorem \ref{thm:cell}(2).  The following result is proved in \cite{Lamnotes}.

\begin{theorem}\label{thm:bridgerecursion}
Suppose $\M$ is not represented by the empty graph.  Let $d  = \dim(\Pi_\M)$.  Then there exists a positroid $\M'$ such that either
\begin{enumerate}
\item $\dim(\Pi_{\M'}) = d-1$ and $N_\M(t_1,t_2,\ldots,t)$ is obtained from $N_{\M'}(t_1,t_2,\ldots,t_{d-1})$ by adding a bridge black at $i$ and white at $i+1$, such that the bridge edge has weight $t$, and all other added edges have weight $1$; or
\item $\dim(\Pi_{\M'}) = d$ and $N_\M(t_1,t_2,\ldots,t_d)$ is obtained from $N_{\M'}(t_1,t_2,\ldots,t_d)$ by inserting a lollipop at some new boundary vertex $i$.
\end{enumerate}
\end{theorem}

If $\dim(\Pi_\M) = 0$, then $\M$ consists of a single subset $I$, and $\Pi_\M$ is the unique point in $\Gr(k,n)$ where all Pl\"ucker variables are 0, except $\Delta_I \neq 0$.  Such a point is represented by a lollipop graph $N$, with white lollipops at the locations specified by $I$.  For example, the planar bipartite graph
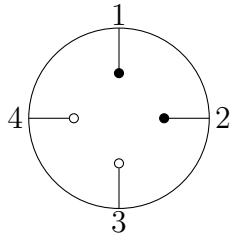
\begin{center}
\begin{tikzpicture}[scale = 0.6]
\node at (0,2.3) {$1$};
\node at (2.3,0) {$2$};
\node at (0,-2.3) {$3$};
\node at (-2.3,0) {$4$};
\draw (0,0) circle (2cm);
\draw (-2,0) -- (-1,0);
\draw (2,0) -- (1,0);
\draw (0,1) -- (0,2);
\draw (0,-1) -- (0,-2);
\filldraw[black] (1,0) circle (0.1cm);
\filldraw[black] (0,1) circle (0.1cm);
\filldraw[white] (0,-1) circle (0.1cm);
\draw (0,-1) circle (0.1cm);
\filldraw[white] (-1,0) circle (0.1cm);
\draw (-1,0) circle (0.1cm);
\end{tikzpicture}
\end{center}
represents such a point with $I = \{3,4\}$.

\section{Temperley-Lieb immanants and the double-dimer model}\label{sec:doubledimer}
\subsection{Double dimers}
\label{ssec:doubledimer}
A \defn{$(k,n)$-partial non-crossing pairing} is a pair $(\tau, T)$ where $\tau$ is a matching of a subset $S = S(\tau) \subset \{1,2,\ldots,n\}$ of even size, such that when the vertices are arranged in order on a circle, and the edges are drawn in the interior, then the edges do not intersect; and $T$ is a subset of $[n] \setminus S$ satisfying $|S| + 2 |T| = 2k$.  Let $\A_{k,n}$ denote the set of $(k,n)$-partial non-crossing pairings.

A subgraph $\Sigma \subset N$ is a \defn{Temperley-Lieb subgraph} if it is a union of connected components each of which is: (a) a path between boundary vertices, or (b) an interior cycle, or (c) a single edge, such that every interior vertex is used.  Let $(\Pi,\Pi')$ be a double-dimer (that is, a pair of dimer configurations) in $N$.  Then the union $\Sigma = \Pi \cup \Pi'$ is a Temperley-Lieb subgraph:

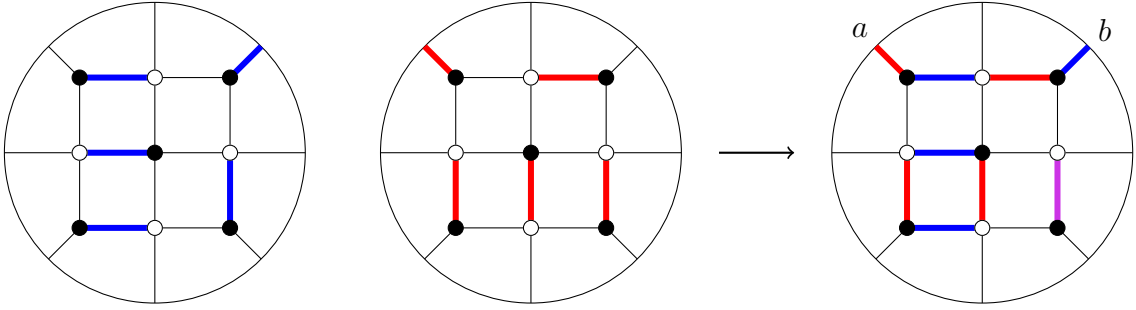
\begin{tikzpicture}
\coordinate (a4) at (45:2);
\coordinate (a3) at (90:2);
\coordinate (a2) at (135:2);
\coordinate (a1) at (180:2);
\coordinate (a8) at (225:2);
\coordinate (a7) at (270:2);
\coordinate (a6) at (315:2);
\coordinate (a5) at (0:2);
\coordinate (x11) at (-1,1);
\coordinate (x12) at (0,1);
\coordinate (x13) at (1,1);
\coordinate (x21) at (-1,0);
\coordinate (x22) at (0,0);
\coordinate (x23) at (1,0);
\coordinate (x31) at (-1,-1);
\coordinate (x32) at (0,-1);
\coordinate (x33) at (1,-1);

\draw (a1) -- (x21);
\draw (a2) -- (x11);
\draw (a3) -- (x12);
\draw (a4) -- (x13);
\draw (a5) -- (x23);
\draw (a6) -- (x33);
\draw (a7) -- (x32);
\draw (a8) -- (x31);
\draw (x11) -- (x13);
\draw (x21) -- (x23);
\draw (x31) -- (x33);
\draw (x11) -- (x31);
\draw (x12) -- (x32);
\draw (x13) -- (x33);
\draw (0,0) circle (2cm);
\draw[blue,line width = 0.08cm] (x11) -- (x12);
\draw[blue,line width = 0.08cm] (x13) -- (a4);
\draw[blue,line width = 0.08cm] (x21) -- (x22);
\draw[blue,line width = 0.08cm] (x31) -- (x32);
\draw[blue,line width = 0.08cm] (x23) -- (x33);

\filldraw[black] (x11) circle (0.1cm);
\filldraw[black] (x13) circle (0.1cm);
\filldraw[black] (x22) circle (0.1cm);
\filldraw[black] (x31) circle (0.1cm);
\filldraw[black] (x33) circle (0.1cm);

\filldraw[white] (x12) circle (0.1cm);
\draw (x12) circle (0.1cm);
\filldraw[white] (x21) circle (0.1cm);
\draw (x21) circle (0.1cm);
\filldraw[white] (x23) circle (0.1cm);
\draw (x23) circle (0.1cm);
\filldraw[white] (x32) circle (0.1cm);
\draw (x32) circle (0.1cm);

\begin{scope}[{shift={(5,0)}}]
\coordinate (a4) at (45:2);
\coordinate (a3) at (90:2);
\coordinate (a2) at (135:2);
\coordinate (a1) at (180:2);
\coordinate (a8) at (225:2);
\coordinate (a7) at (270:2);
\coordinate (a6) at (315:2);
\coordinate (a5) at (0:2);
\coordinate (x11) at (-1,1);
\coordinate (x12) at (0,1);
\coordinate (x13) at (1,1);
\coordinate (x21) at (-1,0);
\coordinate (x22) at (0,0);
\coordinate (x23) at (1,0);
\coordinate (x31) at (-1,-1);
\coordinate (x32) at (0,-1);
\coordinate (x33) at (1,-1);
\draw (a1) -- (x21);
\draw (a2) -- (x11);
\draw (a3) -- (x12);
\draw (a4) -- (x13);
\draw (a5) -- (x23);
\draw (a6) -- (x33);
\draw (a7) -- (x32);
\draw (a8) -- (x31);
\draw (x11) -- (x13);
\draw (x21) -- (x23);
\draw (x31) -- (x33);
\draw (x11) -- (x31);
\draw (x12) -- (x32);
\draw (x13) -- (x33);
\draw (0,0) circle (2cm);
\draw[red,line width = 0.08cm] (a2) -- (x11);
\draw[red,line width = 0.08cm] (x13) -- (x12);
\draw[red,line width = 0.08cm] (x21) -- (x31);
\draw[red,line width = 0.08cm] (x22) -- (x32);
\draw[red,line width = 0.08cm] (x23) -- (x33);

\filldraw[black] (x11) circle (0.1cm);
\filldraw[black] (x13) circle (0.1cm);
\filldraw[black] (x22) circle (0.1cm);
\filldraw[black] (x31) circle (0.1cm);
\filldraw[black] (x33) circle (0.1cm);

\filldraw[white] (x12) circle (0.1cm);
\draw (x12) circle (0.1cm);
\filldraw[white] (x21) circle (0.1cm);
\draw (x21) circle (0.1cm);
\filldraw[white] (x23) circle (0.1cm);
\draw (x23) circle (0.1cm);
\filldraw[white] (x32) circle (0.1cm);
\draw (x32) circle (0.1cm);
\end{scope}

\draw[->,thick] (7.5,0) -- (8.5,0);

\begin{scope}[{shift={(11,0)}}]
\coordinate (a4) at (45:2);
\coordinate (a3) at (90:2);
\coordinate (a2) at (135:2);
\coordinate (a1) at (180:2);
\coordinate (a8) at (225:2);
\coordinate (a7) at (270:2);
\coordinate (a6) at (315:2);
\coordinate (a5) at (0:2);
\coordinate (x11) at (-1,1);
\coordinate (x12) at (0,1);
\coordinate (x13) at (1,1);
\coordinate (x21) at (-1,0);
\coordinate (x22) at (0,0);
\coordinate (x23) at (1,0);
\coordinate (x31) at (-1,-1);
\coordinate (x32) at (0,-1);
\coordinate (x33) at (1,-1);

\node at (135:2.3) {$a$};
\node at (45:2.3) {$b$};

\draw (a1) -- (x21);
\draw (a2) -- (x11);
\draw (a3) -- (x12);
\draw (a4) -- (x13);
\draw (a5) -- (x23);
\draw (a6) -- (x33);
\draw (a7) -- (x32);
\draw (a8) -- (x31);
\draw (x11) -- (x13);
\draw (x21) -- (x23);
\draw (x31) -- (x33);
\draw (x11) -- (x31);
\draw (x12) -- (x32);
\draw (x13) -- (x33);
\draw (0,0) circle (2cm);
\draw[red,line width = 0.08cm] (a2) -- (x11);
\draw[red,line width = 0.08cm] (x13) -- (x12);
\draw[red,line width = 0.08cm] (x21) -- (x31);
\draw[red,line width = 0.08cm] (x22) -- (x32);
\definecolor{mycolor}{rgb}{0.8,0.2,0.9}
\draw[mycolor,line width = 0.08cm] (x23) -- (x33);
\draw[blue,line width = 0.08cm] (x11) -- (x12);
\draw[blue,line width = 0.08cm] (x13) -- (a4);
\draw[blue,line width = 0.08cm] (x21) -- (x22);
\draw[blue,line width = 0.08cm] (x31) -- (x32);

\filldraw[black] (x11) circle (0.1cm);
\filldraw[black] (x13) circle (0.1cm);
\filldraw[black] (x22) circle (0.1cm);
\filldraw[black] (x31) circle (0.1cm);
\filldraw[black] (x33) circle (0.1cm);

\filldraw[white] (x12) circle (0.1cm);
\draw (x12) circle (0.1cm);
\filldraw[white] (x21) circle (0.1cm);
\draw (x21) circle (0.1cm);
\filldraw[white] (x23) circle (0.1cm);
\draw (x23) circle (0.1cm);
\filldraw[white] (x32) circle (0.1cm);
\draw (x32) circle (0.1cm);
\end{scope}

\end{tikzpicture}

The set $S$ of vertices used by the paths on the Temperley-Lieb subgraph is given by $S =(I(\Pi) \setminus I(\Pi')) \cup (I(\Pi') \setminus I(\Pi))$.  Thus each Temperley-Lieb subgraph $\Sigma$ gives a partial non-crossing pairing on $S \subset \{1,2,\ldots,n\}$.  For example, in the above picture we have that $a$ is paired with $b$ and $S = \{a,b\}$.  Note that a  Temperley-Lieb subgraph $\Sigma$ can arise from a pair of matchings in many different ways: it does not remember which edge in a path came from which of the two original dimer configurations.

For each $(k,n)$-partial non-crossing pairing $(\tau,T) \in \A_{k,n}$, define the \defn{Temperley-Lieb immanant}
$$
F_{\tau,T}(N) := \sum_{\Sigma} \wt(\Sigma)
$$
to be the sum over Temperley-Lieb subgraphs $\Sigma$ which give boundary path pairing $\tau$, and $T$ contains black boundary vertices used twice in $\Sigma$, together with white boundary vertices not used in $\Sigma$.  Here $\wt(\Sigma)$ is the product of all weights of edges in $\Sigma$ times $2^{\# {\rm cycles}}$; also, connected components that are single edges have squared weights -- the weight of an edge component in $\Sigma$ is the square of the weight of that edge.  The function $F_{\tau,T}$ is a Grassmann-analogue of Rhoades and Skandera's Temperley-Lieb immanants \cite{RS}.  It would also be reasonable to call these $A_1$-web immanants.

Given $I, J \in \binom{[n]}{k}$, we say that a $(k,n)$-partial non-crossing pairing $(\tau,T)$ is compatible with $I,J$ if:
\begin{enumerate}
\item $S(\tau) = (I \setminus J) \cup (J \setminus I)$, and each edge of $\tau$ matches a vertex in $(I \setminus J)$ with a vertex in $(J \setminus I)$, and
\item 
$T = I \cap J$.
\end{enumerate}

\begin{theorem}\label{thm:TL}
For $I, J \in \binom{[n]}{k}$, we have
$$
\Delta_I(N) \Delta_J(N) = \sum_{\tau,T} F_{\tau,T}(N)
$$
where the summation is over all $(k,n)$-partial non-crossing pairings $\tau$ compatible with $I, J$.
\end{theorem}
\begin{proof}
The only thing left to prove is the compatibility property.

Let $\Pi, \Pi'$ be almost perfect matchings of $N$ such that $I(\Pi) = I$ and $I(\Pi') = J$.  Let $p$ be one of the boundary paths in $\Pi \cup \Pi'$, with endpoints $s$ and $t$.  If $s$ and $t$ have the same color, then the path is even in length.  If $s$ and $t$ have different colors, then the path is odd in length.  In both cases one of $s$ and $t$ belongs to $I \setminus J$ and the other belongs to $J \setminus I$.  
\end{proof}

\subsection{Proof of first statement in Theorem \ref{thm:matchingplucker}}
We shall use the following result.

\begin{proposition}
A non-zero vector $(\Delta_I)_{I \in \binom{[n]}{k}}$ lies in $\Gr(k,n)$ if and only if the Pl\"ucker relation with one index swapped is satisfied:
\begin{equation}\label{eq:plucker}
\sum_{r=1}^k \Delta_{i_1,i_2,\ldots,i_{k-1},j_r}\Delta_{j_1,\ldots,j_{r-1},\hat j_r,j_{r+1},\ldots,j_k} = 0
\end{equation}
where $\hat j_r$ denotes omission.
\end{proposition}
The convention is that $\Delta_I$ is antisymmetric in its indices, so for example $\Delta_{13} = -\Delta_{31}$.

Now use Theorem \ref{thm:TL} to expand \eqref{eq:plucker} with $\Delta_I =\Delta_I(N)$ as a sum of $F_{\tau,T}(N)$ over pairs $(\tau,T)$ (with multiplicity).  We note that the set $T$ is always the same in any term that comes up.  We assume that $i_1 < i_2 < \cdots < i_{k-1}$ and $j_1 < j_2 < \cdots < j_{k+1}$.

So each term $F_{\tau,T}$ is labeled by $(I,J,\tau)$ where $I,J$ is compatible with $\tau$, and $I,J$ occur as a term in \eqref{eq:plucker}.  We provide an involution on such terms.  By the compatibility condition, all but one of the edges in $\tau$ uses a vertex in $\{i_1,i_2,\ldots,i_{k-1}\}$.  The last edge is of the form $(j_a,j_b)$, where $j_a \in I$ and $j_b \in J$.  The involution swaps $j_a$ and $j_b$ in $I, J$ but keeps $\tau$ the same.

Finally we show that this involution is sign-reversing.  Let $I' = I \cup \{j_b\} - \{j_a\}$ and $J' = J \cup \{j_a\} - \{j_b\}$.  Then the sign associated to the term labeled by $(I,J,\tau)$ is equal to $(-1)$ to the power of $\#\{r \in [k] \mid i_r > j_a\} + a$.   Note that by the non-crossingness of the edges in $\tau$ there must be an even number of vertices belonging to $(I \setminus J) \cup (J \setminus I)$ strictly between $j_a$ and $j_b$.  Thus $j_b - j_a = (b-a) + (\#\{r \in [k] \mid i_r > j_b\}-\#\{r \in [k] \mid i_r > j_a\}) \mod 2$ is odd.  So the sign changes.

\subsection{Transition formulae}
So far $F_{\tau,T}$ has been defined as a function of a planar bipartite graph $N$.  

\begin{proposition}\label{prop:tau}
The function $F_{\tau,T}(N)$ depends only on $\tM(N)$ and thus gives a function $F_{\tau,T}$ on $\tGr(k,n)$.
\end{proposition}

To prove this result, one could check the local moves and use Theorem \ref{thm:Pos}.  This is straightforward, and we will do a similar check later for web immanants (Proposition \ref{prop:web}).  Instead, here we will argue somewhat indirectly, by inverting the formula in Theorem \ref{thm:TL}.  

We say that $(I,J)$ is a \defn{standard monomial} if $i_r \leq j_r$ for all $r$ (in other words, $I, J$ form the columns of a semistandard tableau).  
\begin{proposition}\label{prop:TL}
There is a bijection 
$$
\theta:\left\{\text{standard monomials in } \binom{[n]}{k}\right\} \longrightarrow \A_{k,n},$$ 
and a partial order $\leq$ on standard monomials such that the transition matrix between $\{\Delta_I(N) \Delta_J(N) \mid (I,J) \text{ standard} \}$ and $\{F_{\tau,T} \mid (\tau,T) \in \A_{k,n}\}$ is unitriangular.  More precisely, 
$$
\Delta_{I}(N) \Delta_J(N) =  F_{\theta(I,J)}(N) + \sum_{(I',J') < (I,J)} F_{\theta(I',J')}(N).
$$
\end{proposition}
\begin{proof}
Since the subset $T = I \cap J$ plays little role, we shall assume $T = \emptyset$, and for simplicity, $I \cup J = [n]$.

Then $(I,J)$ is a two-column tableaux using the number $1,2,\ldots,2k=n$.  The bijection $\theta$ sends such $I,J$ to the non-crossing pairing $\tau$ on $[2k]$ given by connecting $i_r$ to $j_s > i_r$ where $s$ is chosen minimal so that $\#(I \cap (j_s-i_r)) = \#(J \cap (j_s-i_r))$.  This bijection can be described in terms of Dyck paths as follows: draw a Dyck path $P_{I,J}$ having a diagonally upward edge $E_i$ at positions specified by $i \in I$ and a diagonally downward edge $D_j$ at positions specified by $j \in J$.  Then $\tau$ joins $i$ to $j$ if the horizontal rightwards ray starting at $E_i$ intersects $D_j$ before it intersects any other edge.

The partial order $\leq$ is a refinement of the partial order: $(I',J') \prec (I,J)$ if the Dyck path $P_{I,J}$ stays weakly below $P_{I',J'}$ the entirety of the path.  To see this, suppose $P_{I,J}$ goes above $P_{I',J'}$ somewhere.  Let $a$ be the first position this happens.  Then $a$ is an up step in $P_{I,J}$ (that is $a \in I$) and a down step in $P_{I',J'}$ (that is, $a \in J'$).   Suppose $a$ is paired with $a' < a$ in $\theta(I',J')$.  Then the edges at positions $a'$ and $a$ are at the same height in $P_{I',J'}$.  Since $P_{I,J}$ is weakly below $P_{I',J'}$ at position $a'$, it follows that the edges at positions $a'$ and $a$ are at different heights in $P_{I,J}$.  So the cardinalities $|I \cap (a',a)|$ and $|J \cap (a',a)|$ differ, and thus pairing $a'$ and $a$ is not compatible with $(I,J)$.
\end{proof}

Thus $F_{\tau,T}(N)$ can be expressed in terms of the Pl\"ucker coordinates $\Delta_I(N)$.  It follows that $F_{\tau,T}$ are functions on $\tGr(k,n)$ and Proposition \ref{prop:tau} follows.  Since $\{\Delta_I \Delta_J \mid (I,J) \text{ standard}\}$ form a basis for the degree two part of the homogeneous coordinate ring $\C[\Gr(k,n)]$ in the Pl\"ucker embedding, we have

\begin{corollary}
The set $\{F_{\tau,T}\}$ forms a basis for the degree two part of the homogeneous coordinate ring of the $\Gr(k,n)$ in the Pl\"ucker embedding.  In particular, the number $|\A_{k,n}|$ of $(k,n)$-non-crossing pairings is equal to the number of semistandard tableaux of shape $2^k$ filled with numbers $1,2,\ldots,n$.
\end{corollary}

\subsection{Restriction to positroid varieties}
The following proposition is a case-by-case check.  

\begin{proposition}\label{prop:bridge}
Let $N$ be obtained from $N'$ by adding a bridge black at $i$ to white at $i+1$ with bridge edge having weight $t$.  Then
$$
F_{\tau,T}(N) =
 \begin{cases}  tF_{\tau-(i,i+1),T \cup \{i+1\}} + F_{\tau,T} & \mbox{$(i,i+1)$ is in  $\tau$} \\
F_{\tau,T} & \mbox{$(i,a)$ and $(i+1,b)$ are in $\tau$} \\
F_{\tau,T}  & \mbox{$(i,a)$ is in $\tau$ and $i+1 \in T$} \\
F_{\tau,T} + tF_{\tau - (i,a) \cup (i+1,a),T} & \mbox{$(i,a)$ is in $\tau$ but $i+1 \notin S \cup T$} \\
F_{\tau,T} + tF_{\tau - (i+1,b) \cup (i,b),T} & \mbox{$(i+1,b)$ is in $\tau$ and $i \in T$} \\
F_{\tau,T} & \mbox{$(i+1,b)$ is in $\tau$ but $i \notin S \cup T$} \\
F_{\tau,T} & \mbox{neither $i$ nor $i+1$ is in $\tau$, and $i \notin T$ or $i+1 \in T$} \\
 \end{cases}
$$
and
$$
F_{\tau,T}(N) = t^2 F_{\tau,T-\{i\}\cup\{i+1\}} + t\sum_{(a,b) \in \tau} F_{\tau-(a,b) \cup (i,a) \cup (i+1,b), T-\{i\}} +2tF_{\tau \cup (i,i+1),T-\{i\}} + F_{\tau,T}
$$
if neither $i$ nor $i+1$ is in $\tau$, and $i \in T$ but $i+1 \notin T$.  Here $F_{\tau,T} = F_{\tau,T}(N')$.
\end{proposition}

\begin{remark}
The Lie group $GL(n)$ acts on $\Gr(k,n)$.  Since adding a bridge corresponds to acting by a one parameter subgroup $x_i(t) = \exp(t e_i)$ (see \cite{Lamnotes}), Proposition \ref{prop:bridge} determines the infinitesimal action of the Chevalley generators of $\mathfrak{gl(n)}$ on the functions $F_{\tau,T}$.
\end{remark}

Define
$$
\A(N) := \{(\tau,T) \in \A_{k,n} \mid F_{\tau,T}(N) \neq 0\}.
$$
Let $\M$ be a positroid of rank $k$ on $[n]$.  Let $N$ be a planar bipartite graph representing $\M$.  Then we define $\A(\M) := \A(N)$.  

\begin{lemma}
$\A(\M)$ does not depend on the choice of $N$.
\end{lemma}
\begin{proof}
By Proposition \ref{prop:tau}, $F_{\tau,T}(N)$ depends (up to some global scalar) only on the point $M(N) \in \Gr(k,n)_{\geq 0}$ representing $N$.  Also $\A(\M)$ does not depend on the weights of $N$ chosen, only the underlying unweighted bipartite graph.  The result then follows from Theorem \ref{thm:cell}.
\end{proof}

\begin{remark}\label{rem:TL}
The subset $\A(\M) \subset \A_{k,n}$ is a ``degree two'' analogue of the positroid $\M$.  It would be interesting to give a description of $\A(\M)$ that does not depend on a choice of $N$, similar to Oh's theorem \cite{Oh} characterizing $\M$.
\end{remark}

\begin{theorem}\label{thm:TLbasis}
The set
$$
\{F_{\tau,T} \mid (\tau,T) \in \A(\M)\}
$$
is a basis for the space for the degree two component of the homogeneous coordinate ring $\C[\Pi_\M]$.
\end{theorem}

In other words, the functions $F_{\tau,T}$ either restrict to $0$ on $\Pi_\M$, or they form part of a basis.

\begin{proof}
It is known that every element of $\C[\Pi_\M]$ is obtained from restriction from $\C[\Gr(k,n)]$ (see \cite{KLS}).  So certainly $\{F_{\tau,T} \mid (\tau,T) \in \A_{k,n}\}$ span the stated space.  So it suffices to show that $\{F_{\tau,T} \mid (\tau,T) \in \A(\M)\}$ is linearly independent.

We proceed by induction first on $n$ and then on the dimension of $\Pi_\M$.  The claim is trivially true when $\Pi_\M$ is a point.

Let $\M$ be a positroid.  By the bridge-lollipop recursion (Theorem \ref{thm:bridgerecursion}), either
\begin{enumerate}
\item
a plabic graph $N$ for $\M$ contains a lollipop, or
\item
a plabic graph $N$ for $\M$ is obtained from a plabic graph $N'$ for $\M'$ by adding a bridge, where $\dim(\Pi_{\M'}) = \dim(\Pi_\M) - 1$.
\end{enumerate}
In the first case, let $N'$ be the plabic graph with $n-1$ boundary vertices where a lollipop has been removed.  The inductive hypothesis for $\M(N')$ immediately gives the claim for $\M$.

In the second case, let $d = \dim(\Pi_{\M'})$.  Then a dense subset of $\Pi_{\M'}$ can be parametrized by assigning weights $t_1,t_2,\ldots,t_d$ to $d$ of the edges of $N'$.  By the inductive hypothesis, the functions $\{F_{\tau,T}(N') \mid (\tau,T) \in \A(\M')\}$ are then linearly independent polynomials in $t_1,t_2,\ldots,t_d$.  Let $V$ denote the span of these polynomials.  We may assume that $N$ is obtained from $N'$ by adding a bridge black at $i$ to white at $i+1$ with weight $t$, allowing us to use Proposition \ref{prop:bridge}.

Note that
$$
\{F_{\tau,T}(N) \mid (\tau,T) \in \A(\M)\}
$$
can then be thought of as a set of polynomials in $t$, with coefficients in $V$.  We need to show that these polynomials $p_{\tau,T}(t) = F_{\tau,T}(N)$ are linearly independent.  Suppose there exists a linear relation
$$
\sum_{(\tau,T)\in \A(\M)} a_{\tau,T} p_{\tau,T}(t) = 0.
$$
Then we will get a linear relation for each of the coefficients of $t^2,t,1$.  Consider first the linear relation for the constant coefficient.  By Proposition \ref{prop:bridge}, we get
$$
0 = \sum_{(\tau,T) \in \A(\M')} a_{\tau,T} [t^0] p_{\tau,T}(t) = \sum_{(\tau,T) \in A(\M')} a_{\tau,T} F_{\tau,T}(N').
$$
By the inductive hypothesis, we see that $a_{\tau,T} = 0$ for $(\tau,T) \in \A(\M')$.

Now let us write
$$
[t^1]\sum_{(\tau,T)\in \A(\M)} a_{\tau,T} p_{\tau,T}(t) = \sum_{(\kappa,R) \in \A(\M')} b_{\kappa,R} F_{\kappa,R}(N').$$  
By Proposition \ref{prop:bridge},
$$
b_{\kappa,R} = \begin{cases} a_{\kappa \cup (i,i+1), R - \{i+1\}} & \mbox{if $i$ and $i+1$ are not in $\kappa$ and $i+1 \in R$} \\
a_{\kappa \cup (i,a)- (i+1,a), R} & \mbox{if $(i+1,a) \in \kappa$ and $i+1 \notin S(\kappa) \cup R$} \\

a_{\kappa - (i,b)\cup (i+1,b), R} &
\mbox{if $(i,b) \in \kappa$ and $i \in S(\kappa)$} \\
 \frac{1}{2} a_{\kappa -(i,i+1), R \cup \{i\}}& \mbox{if $(i,i+1) \in \tau$ and $i ,i+1 \notin R$} \\
a_{\kappa \cup (a,b) -(i,a)- (i+1,b), R \cup \{i\}}& \mbox{if $(i,a), (i+1,b) \in \tau$ and $i ,i+1 \notin R$}
\end{cases}
$$
It follows from this that $a_{\tau,T}$ has to be 0 if the coefficient of $t$ in $p_{\tau,T}(t)$ is non-zero.  Similarly, $a_{\tau,T}$ is 0 if the coefficient of $t^2$ in $p_{\tau,T}(t)$ is non-zero.  But by definition, one of the three coefficients of $p_{\tau,T}(t)$ is non-zero when $(\tau,T) \in \A(\M)$.  It follows that the stated polynomials are linearly independent.
\end{proof}

A basis of standard monomials for $\C[\Pi_\M]$ follows from the methods of \cite{KLS,LaLi}.  Let $u \leq v$ be an interval in Bruhat order such that the Richardson subvariety $X^v_u$ of the flag variety projects birationally to $\Pi_\M$ (see \cite[Theorem 5.1]{KLS}).  Call an ordered pair $(I,J)$ of $k$-element subsets \defn{$\M$-standard} if $I = \pi_k(x)$ and $J = \pi_k(y)$ where $u \leq x \leq y \leq v$.  Then by \cite[Proposition 7.2]{KLS} we have that
$$
\{\Delta_I \Delta_J \mid (I,J) \text{ is $\M$-standard}\}
$$
forms a basis for the degree two part of $\C[\Pi_\M]$.   This suggests the following conjecture.

\begin{conj}\label{conj:TL}
The map $\theta$ of Proposition \ref{prop:TL} restricts to a bijection between $\M$-standard pairs $(I,J)$ and the set $\A(\M)$.
\end{conj}

\subsection{Pl\"ucker coordinates for the TNN Grassmannian are ``log-concave''}\label{sec:logconcave}

Let $I = \{i_1<i_2<\ldots,i_k\}, J=\{j_1<\cdots<j_k\} \in \binom{[n]}{k}$.  Suppose the multiset $I \cup J$, when sorted, is equal to $\{a_1 \leq b_1 \leq a_2 \leq \cdots \leq a_k \leq b_k\}$.  Then we define $\sort_1(I,J) := \{a_1,\ldots,a_k\}$ and $\sort_2(I,J) := \{b_1,\ldots,b_k\}$.  Also if $I \cap J = \emptyset$ define $\min(I,J) := \{\min(i_1,j_1),\ldots, \min(i_k,j_k)\}$ and similarly $\max(I,J)$; for general $I, J$, we define $\min(I,J) := \min(I-J, J-I) \cup (I \cap J)$ and similarly for $\max(I,J):=\max(I-J, J-I) \cup (I \cap J)$.

The following result was independently obtained by Farber and Postnikov \cite{FaPo}.  It says that the Pl\"ucker coordinates $\Delta_I(X)$ of $X \in \tGr(k,n)_{\geq 0}$ are log-concave.  See \cite{LPP} for a related situation in Schur positivity.

\begin{proposition}\label{prop:inequalities}
Let $X \in \tGr(k,n)_{\geq 0}$.  Then 
$$
\Delta_I(X) \Delta_J(X) \leq \Delta_{\min(I,J)}(X) \Delta_{\max(I,J)}(X) \leq \Delta_{\sort_1(I,J)}(X) \Delta_{\sort_2(I,J)}(X).
$$
\end{proposition}
\begin{proof}
Follows from Theorem \ref{thm:TL} and an analysis of compatibility.
\end{proof}

A matroid $\M$ is \defn{sort-closed} if $I,J \in \M$ implies $\sort_1(I,J), \sort_2(I,J) \in \M$.  We deduce the following result, first proved by Lam and Postnikov \cite{LP}, in the context of alcoved polytopes \cite{LPalcoved}.

\begin{corollary}\label{cor:sortclosed}
Positroids are sort-closed.
\end{corollary}

In fact the converse of Corollary \ref{cor:sortclosed} also holds \cite{LP}: a sort-closed matroid is a positroid.

\section{Webs and triple dimers}
\subsection{$A_2$-webs and reductions}
We review Kuperberg's $A_2$-webs, modified for our situation by allowing tagged boundary vertices.

As usual an integer $n$ is fixed.  A \defn{web} is a planar bipartite graph embedded into a disk where all interior vertices are trivalent.  Edges are always directed towards white interior vertices and away from black interior vertices.  Furthermore, we allow some vertex-less directed cycles in the interior, some directed edges from one boundary vertex to another, and some boundary vertices that are otherwise not used to be ``tagged".

Note that boundary vertices of a web are not colored.  In the following picture, the boundary vertex $6$ is tagged but $7$ is not.

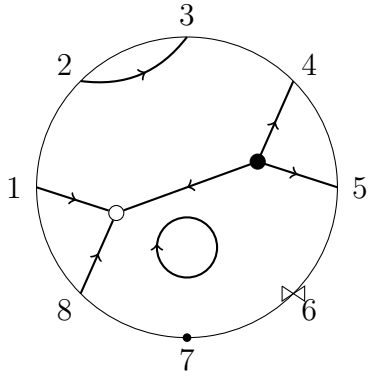
\begin{center}
\begin{tikzpicture}
\coordinate (a4) at (45:2);
\coordinate (a3) at (90:2);
\coordinate (a2) at (135:2);
\coordinate (a1) at (180:2);
\coordinate (a8) at (225:2);
\coordinate (a7) at (270:2);
\coordinate (a6) at (315:2);
\coordinate (a5) at (0:2);
\coordinate (X) at (200:1);
\coordinate (Y) at (20:1);

\node at (45:2.3) {$4$};
\node at (90:2.3) {$3$};
\node at (135:2.3) {$2$};
\node at (180:2.3) {$1$};
\node at (225:2.3) {$8$};
\node at (270:2.3) {$7$};
\node at (315:2.3) {$6$};
\node at (0:2.3) {$5$};

\draw (0,0) circle (2);

\draw[thick,->-] (a2) to [bend right] (a3);

\begin{scope}[decoration={
    markings,
    mark=at position 0.5 with {\arrow{>}}}
    ] 

\draw[thick,postaction={decorate}] (a1) -- (X);
\draw[thick,postaction={decorate}] (a8) -- (X);
\draw[thick,postaction={decorate}] (Y) -- (X);
\draw[thick,postaction={decorate}] (Y) -- (a4);
\draw[thick,postaction={decorate}] (Y) -- (a5);
\end{scope}

\filldraw [white] (X) circle (0.1cm);
\filldraw [black] (Y) circle (0.1cm);
\filldraw [black] (a7) circle (0.5mm);
\begin{scope}[{shift={(a6)}}]
\draw (-0.15,0.1) -- (0.15,-0.1) -- (0.15,0.1) -- (-0.15,-0.1) -- (-0.15,0.1);
\end{scope}
\draw (X) circle (0.1cm);
\draw[thick] (0,-0.8) circle (0.4cm);
\draw[->,thick] (-0.4,-0.8) -- (-0.4,-0.8);
\end{tikzpicture}
\end{center}

The \defn{degree} $d(W)$ of a web $W$ is given by
$$
d(W) = 3\#\{\text{boundary tags}\} + 3 \#\{\text{boundary paths}\} $$
$$+ \#\{\text{boundary vertices incident to a white interior vertex}\}
$$
$$ +2 \#\{\text{boundary vertices incident to a black interior vertex}\}.
$$
A simple counting argument shows that $d(W)$ is always divisible by 3.  In the above example we get $d(W) = 12 = 3 \times 4$.  Let $\W_{k,n}$ denote the (infinite) set of webs $W$ on $n$ boundary vertices, satisfying $d(W) = 3k$.  

A web $W$ is called \defn{non-elliptic} if it has no contractible loops, no pairs of edges enclosing a contractible disk, and no simple $4$-cycles all of whose vertices are internal and which enclose a contractible disk.  (Here contractible means that the enclosed region contains no other edges of the graph.)  Let $\D_{k,n}$ denote the set of non-elliptic webs $D$ on $n$ boundary vertices, satisfying $d(D) = 3k$.

Any web $W$ can be reduced to a formal (but finite) linear combination of non-elliptic webs, using the rules:
\
\begin{enumerate}
\item For either orientation,
$$
\begin{tikzpicture}
\draw (0,0) circle (0.5cm);
\node at (1.5,0) {$ = 3$};
\end{tikzpicture}
$$
\item
$$
\begin{tikzpicture}
\draw (-1.3,0) -- (-0.5,0);
\draw (1.3,0) -- (0.5,0);
\draw (0,0) circle (0.5cm);
\node at (1.8,0) {$ = 2$};
\draw[->-] (4.5,0) -- (2.5,0);
\filldraw[black] (0.5,0) circle (0.1cm);
\filldraw[white] (-0.5,0) circle (0.1cm);
\draw (-0.5,0) circle (0.1cm);
\end{tikzpicture}
$$
\item
$$
\begin{tikzpicture}[scale=0.8]
\coordinate (a) at (-0.5,0.5);
\coordinate (b) at (0.5,0.5);
\coordinate (c) at (0.5,-0.5);
\coordinate (d) at (-0.5,-0.5);
\draw (a) -- (b) -- (c) -- (d) -- (a);
\draw (a) -- (-1,1);
\draw (b) -- (1,1);
\draw (c) -- (1,-1);
\draw (d) -- (-1,-1);
\filldraw[black] (a) circle (0.1cm);
\filldraw[white] (b) circle (0.1cm);
\draw (b) circle (0.1cm);
\filldraw[black] (c) circle (0.1cm);
\filldraw[white] (d) circle (0.1cm);
\draw (d) circle (0.1cm);
\node at (2,0) {$ = $};
\begin{scope}[{shift={(4,0)}}]
\draw[->-] (1,1) to [bend left] (-1,1);
\draw[->-] (-1,-1) to [bend left] (1,-1);
\end{scope}
\node at (6,0) {$ + $};
\begin{scope}[{shift={(8,0)}}]
\draw[->-] (1,1) to [bend right](1,-1);
\draw[->-] (-1,-1) to [bend right] (-1,1);
\end{scope}
\end{tikzpicture}
$$
\end{enumerate}

Note that our signs differ somewhat from Kuperberg's, but agrees with those of Pylyavskyy \cite{Pyl}.  Kuperberg \cite{Kup,Kup2} shows that this reduction process is confluent: we get an expression
$$
W = \sum_{D \in \D_{k,n}} W_D D 
$$
expressing a web $W$ in terms of non-elliptic webs $D$, where the coefficients $W_D \in \Z$ do not depend on the choices of reduction moves performed.  

\subsection{Weblike subgraphs}
Let $G \subset N$ be a subgraph consisting of 
\begin{enumerate}
\item
some connected components $A_1,A_2,\ldots,A_r$ where every vertex is either (a) internal trivalent, (b) internal bivalent, or (c) a boundary leaf, and such that if $v$ and $w$ are two trivalent vertices connected by a path consisting only of bivalent vertices, then $v$ and $w$ have different colors (or equivalently, the path between $v$ and $w$ has an odd number of edges), and
\item
some (internal) simple cycles $C_1,C_2,\ldots,C_s$ (necessarily of even length), and
\item
some isolated edges $E_1,E_2,\ldots,E_t$ (dipoles).
\end{enumerate}
Furthermore, we require that any component $A_i$ that has no trivalent vertices (and is thus a path between boundary vertices) is equipped with an orientation.  We call such a subgraph $G$ that uses all the internal vertices a \defn{weblike} subgraph.

In the following, we shall abuse notation by using $e$ to both denote an edge $e$, and the weight of that same edge.  To each subgraph $G$ we associate the weight
$$
\wt(G) := \prod_{i=1}^r \wt(A_i) \; \prod_{j=1}^s \wt(C_j) \; \prod _{\ell=1}^t \wt(E_\ell)
$$
where
\begin{enumerate}
\item If $A_i$ is a path between boundary vertices consisting of edges $e_1,e_2,\ldots,e_d$ (listed in order of the orientation), then
$$
\wt(A_i)=
\begin{cases} e_1e_2^2e_3e_4^2\cdots & \mbox{if the first internal vertex along $A_i$ is white,} \\
e_1^2e_2e_3^2e_4\cdots & \mbox{if the first internal vertex along $A_i$ is black.}
\end{cases}
$$
\item
If $A_i$ is not a path, then 
$$
\wt(A_i) = \prod_{e \in A_i} e^{a_e}
$$
and 
$$
a_e = \begin{cases} 1 & \mbox{if $e$ is incident to, or an even distance from, a trivalent vertex,} \\
2 & \mbox{if $e$ is an odd distance from a trivalent vertex.}
\end{cases}
$$
\item
If $C_i$ is a cycle with edges $e_1,e_2,\ldots,e_{2m}$ in cyclic order, then
$$
\wt(C_i) = e_1e_2^2\cdots e_{2m-1} e_{2m}^2 + e_1^2e_2\cdots e_{2m-1}^2 e_{2m}.
$$
\item
If $E_i$ is an isolated edge $e$, then 
$$
\wt(E_i) = e^3.
$$
\end{enumerate}

To a weblike graph $G$, we associate a web $W = W(G)$ as follows: 
\begin{enumerate}
\item
Each component $A_i$ gives rise to a component $W_i$ obtained by removing all bivalent vertices, and orienting all edges towards the white internal trivalent vertices.  In the case that $A_i$ has no internal trivalent vertex we orient the edge using the orientation of the path in $G$.
\item
Each cycle $C_i$ is replaced by a vertexless loop oriented arbitrarily.
\item
All internal edges $E_i$ are removed; a black boundary  vertex (that is, a boundary vertex that is adjacent to a white interior vertex) is ``tagged" if it belongs to an edge $E_i$; a white boundary vertex (that is, a boundary vertex that is adjacent to a black interior vertex) is ``tagged" if it is not used in $G$.
\end{enumerate}

We consider a boundary vertex $i$ to be used in $D$ if it belongs to a component that contains edges.  Thus boundary vertices that are tagged are not considered used.

\begin{lemma}\label{lem:webrepresent}
Suppose $G \subset N$ is a weblike subgraph with web $W = W(G)$.  Suppose $W'$ is some other web that can be obtained from $W'$ by a series of reductions.  Then there exists a weblike subgraph $G' \subset N$ such that $W(G') = W'$.
\end{lemma}
\begin{proof}
A reduction $W \mapsto W'$ corresponds to removing some of the edges in $W$ (and then removing bivalent vertices that result).  This can be achieved on the level of weblike subgraphs by replacing a path of odd length by some isolated dipoles:
\begin{center}
\begin{tikzpicture}
\coordinate (b) at (0,0);
\coordinate (w) at (1,0);
\coordinate (b2) at (2,0);
\coordinate (w2) at (3,0);
\coordinate (b3) at (4,0);
\coordinate (w3) at (5,0);
\draw (b) -- (w3);
\filldraw [white] (w) circle (0.1cm);
\filldraw [black] (b) circle (0.1cm);
\draw (w) circle (0.1cm);
\filldraw [black] (b2) circle (0.1cm);
\filldraw [white] (w2) circle (0.1cm);
\draw (w2) circle (0.1cm);
\filldraw [black] (b3) circle (0.1cm);
\filldraw [white] (w3) circle (0.1cm);
\draw (w3) circle (0.1cm);
\draw [->] (6,0) -- (7,0);

\begin{scope}[shift={(8,0)}]
\coordinate (b) at (0,0);
\coordinate (w) at (1,0);
\coordinate (b2) at (2,0);
\coordinate (w2) at (3,0);
\coordinate (b3) at (4,0);
\coordinate (w3) at (5,0);
\draw (w) -- (b2);
\draw (w2) -- (b3);
\filldraw [white] (w) circle (0.1cm);
\filldraw [black] (b) circle (0.1cm);
\draw (w) circle (0.1cm);
\filldraw [black] (b2) circle (0.1cm);
\filldraw [white] (w2) circle (0.1cm);
\draw (w2) circle (0.1cm);
\filldraw [black] (b3) circle (0.1cm);
\filldraw [white] (w3) circle (0.1cm);
\draw (w3) circle (0.1cm);
\end{scope}
\end{tikzpicture}
\end{center}
The same trick allows us to replace an even cycle by a number of isolated dipoles.
\end{proof}

\subsection{Web immanants}

For each non-elliptic web $D \in \D_{k,n}$, we define a generating function, called the \defn{web immanant}
$$
F_D(N) := \sum_W W_D \sum_{W(G) = W} \wt(G).
$$
In otherwords, each subgraph $G$ contributes a multiple of $\wt(G)$ to $F_D$, where the multiple is equal to the coefficient of $D$ in the web $W(G)$.  

\begin{example}\label{ex:big}
We compute $F_D(N)$ for the planar bipartite graph
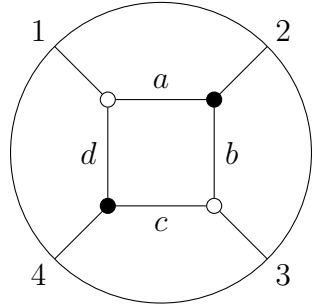
\begin{center}
\begin{tikzpicture}
\coordinate (b) at (45:1);
\coordinate (w) at (135:1);
\coordinate (b1) at (225:1);
\coordinate (w1) at (315:1);
\node at (135:2.3) {$1$} ;
\node at (45:2.3) {$2$} ;
\node at (315:2.3){$3$} ;
\node at (225:2.3){$4$} ;
\draw (0,0) circle (2);
\draw (w) -- node[above]{$a$} (b) --node[right]{$b$} (w1) --node[below]{$c$} (b1) --node[left]{$d$} (w);
\draw (w) -- (135:2);
\draw (b) -- (45:2);
\draw (w1) --  (315:2);
\draw (b1) --  (225:2);
\filldraw [white] (w) circle (0.1cm);
\filldraw [black] (b) circle (0.1cm);
\draw (w) circle (0.1cm);
\filldraw [white] (w1) circle (0.1cm);
\filldraw [black] (b1) circle (0.1cm);
\draw (w1) circle (0.1cm);
\end{tikzpicture}
\end{center}
In the following table we often list tagged boundary vertices as a subset.  There are $|\D_{2,4}| = 50$ non-elliptic webs in this case. 

\newpage

\begin{multicols}{2}
\begin{center}
\begin{tabular}{|c|c|}
\hline 
$D$ & $F_D(N)$ \\
\hline 
$\{1,2\}$&  $b^3$\\
\hline 
$\{1,3\}$&  $1$\\
\hline 
$\{1,4\}$&  $d^3$\\
\hline 
$\{2,3\}$&  $a^3$\\
\hline
$\{2,4\}$ & $3(a^2bc^2d + ab^2cd^2) + a^3c^3 + b^3d^3$\\
\hline 
$\{3,4\}$&  $c^3$\\
\hline
\hline
$(1 \to 2) \cup (3 \to 4)$ & $a^2c^2 + abcd$ \\
\hline
$(2 \to 1) \cup (3 \to 4)$ & $ac^2$\\
\hline
$(1 \to 2) \cup (4 \to 3)$ & $a^2c$ \\
\hline
$(2 \to 1) \cup (4 \to 3)$ & $ac$\\
\hline
$(1 \to 4) \cup (2 \to 3)$ & $bd^2$\\
\hline
$(4 \to 1) \cup (2 \to 3)$ & $bd$\\
\hline
$(1 \to 4) \cup (3 \to 2)$ & $b^2d^2 + abcd$\\
\hline
$(4 \to 1) \cup (3 \to 2)$ & $b^2d$\\
\hline
\hline
$\{1\} \cup (3 \to 4)$ & $c^2$ \\
\hline
$\{1\} \cup (4 \to 3)$ & $c$ \\
\hline
$\{1\} \cup (2 \to 3)$ & $b$ \\
\hline
$\{1\} \cup (3 \to 2)$ & $b^2$ \\
\hline
$\{1\} \cup (2 \to 4)$ & $bc^2$ \\
\hline
$\{1\} \cup (4 \to 2)$ & $b^2c$ \\
\hline
\hline
$\{2\} \cup (3 \to 4)$ & $a^3c^2 + 2a^2bcd + ab^2d^2$ \\
\hline
$\{2\} \cup (4 \to 3)$ & $a^3c + a^2bd$ \\
\hline
$\{2\} \cup (1 \to 4)$ & $b^3d^2 + 2ab^2cd + a^2bc^2$ \\
\hline
$\{2\} \cup (4 \to 1)$ & $b^3d + ab^2c$ \\
\hline
$\{2\} \cup (1 \to 3)$ & $a^2b$ \\
\hline
$\{2\} \cup (3 \to 1)$ & $ab^2$ \\
\hline
\hline
$\{3\} \cup (1 \to 2)$ & $a^2$ \\
\hline
$\{3\} \cup (2 \to 1)$ & $a$ \\
\hline
$\{3\} \cup (4 \to 1)$ & $d$ \\
\hline
$\{3\} \cup (1 \to 4)$ & $d^2$ \\
\hline
$\{3\} \cup (4 \to 2)$ & $a^2d$ \\
\hline
$\{3\} \cup (2 \to 4)$ & $ad^2$ \\
\hline
\hline
$\{4\} \cup (1 \to 2)$ & $a^2c^3 + 2abc^2d + b^2cd^2$ \\
\hline
$\{4\} \cup (2 \to 1)$ & $ac^3 + bc^2d$ \\
\hline
$\{4\} \cup (3 \to 2)$ & $b^2d^3 + 2abcd^2+a^2c^2d$ \\
\hline
$\{4\} \cup (2 \to 3)$ & $bd^3 +acd^2$ \\
\hline
$\{4\} \cup (3 \to 1)$ & $c^2d$ \\
\hline
$\{4\} \cup (1 \to 3)$ & $cd^2$ \\
\hline
\end{tabular}

\columnbreak

\begin{tabular}{|c|c|}
\hline
$\begin{tikzpicture}[scale=0.3]
\coordinate (b1) at (315:1);
\coordinate (a2) at (45:2) ;
\coordinate (a1) at (135:2) ;
\coordinate (a3) at (315:2) ;
\coordinate (a4) at (225:2) ;
\node at (45:2.6) {$2$};
\node at (135:2.6) {$1$} ;
\node at (315:2.6){$3$} ;
\node at (225:2.6){$4$} ;
\draw (0,0) circle (2);
\draw[->-] (b1) -- (a2);
\draw[->-] (b1) -- (a3);
\draw[->-] (b1) -- (a4);
\filldraw [black] (b1) circle (0.2cm);
 \end{tikzpicture}$ & $abd^2 + a^2cd$ \\
\hline
$\begin{tikzpicture}[scale=0.3]
\coordinate (b1) at (225:1);
\coordinate (a2) at (45:2) ;
\coordinate (a1) at (135:2) ;
\coordinate (a3) at (315:2) ;
\coordinate (a4) at (225:2) ;
\node at (45:2.6) {$2$};
\node at (135:2.6) {$1$} ;
\node at (315:2.6){$3$} ;
\node at (225:2.6){$4$} ;
\draw (0,0) circle (2);
\draw[->-] (b1) --  (a3);
\draw[->-] (b1) --  (a4);
\draw[->-] (b1) --  (a1);
\filldraw [black] (b1) circle (0.2cm);
 \end{tikzpicture}$ & $cd$ \\
\hline
$\begin{tikzpicture}[scale=0.3]
\coordinate (b1) at (135:1);
\coordinate (a2) at (45:2) ;
\coordinate (a1) at (135:2) ;
\coordinate (a3) at (315:2) ;
\coordinate (a4) at (225:2) ;
\node at (45:2.6) {$2$};
\node at (135:2.6) {$1$} ;
\node at (315:2.6){$3$} ;
\node at (225:2.6){$4$} ;
\draw (0,0) circle (2);
\draw[->-] (b1) -- (a1);
\draw[->-] (b1) -- (a2);
\draw[->-] (b1) -- (a4);
\filldraw [black] (b1) circle (0.2cm);
 \end{tikzpicture}$ & $b^2cd + abc^2$ \\
\hline
$\begin{tikzpicture}[scale=0.3]
\coordinate (b1) at (45:1);
\coordinate (a2) at (45:2) ;
\coordinate (a1) at (135:2) ;
\coordinate (a3) at (315:2) ;
\coordinate (a4) at (225:2) ;
\node at (45:2.6) {$2$};
\node at (135:2.6) {$1$} ;
\node at (315:2.6){$3$} ;
\node at (225:2.6){$4$} ;
\draw (0,0) circle (2);
\draw[->-] (b1) -- (a2);
\draw[->-] (b1) -- (a3);
\draw[->-] (b1) -- (a1);
\filldraw [black] (b1) circle (0.2cm);
 \end{tikzpicture}$ & $ab$ \\

\hline
$\begin{tikzpicture}[scale=0.3]
\coordinate (w) at (315:1);
\coordinate (a2) at (45:2) ;
\coordinate (a1) at (135:2) ;
\coordinate (a3) at (315:2) ;
\coordinate (a4) at (225:2) ;
\node at (45:2.6) {$2$};
\node at (135:2.6) {$1$} ;
\node at (315:2.6){$3$} ;
\node at (225:2.6){$4$} ;
\draw (0,0) circle (2);
\draw[->-] (a2) -- (w);
\draw[->-] (a3) -- (w);
\draw[->-] (a4) -- (w);
\begin{scope}[{shift={(a1)},scale=2}]
\draw (-0.15,0.1) -- (0.15,-0.1) -- (0.15,0.1) -- (-0.15,-0.1) -- (-0.15,0.1);
\end{scope}
\filldraw [white] (w) circle (0.2cm);
\draw (w) circle (0.2cm);
 \end{tikzpicture}$ & $bc$ \\
\hline
$\begin{tikzpicture}[scale=0.3]
\coordinate (w) at (225:1);
\node at (45:2.6) {$2$};
\node at (135:2.6) {$1$} ;
\node at (315:2.6){$3$} ;
\node at (225:2.6){$4$} ;
\draw (0,0) circle (2);
\draw[->-]  (225:2) -- (w);
\draw[->-] (315:2) -- (w);
\draw[->-] (135:2) -- (w);
\filldraw [white] (w) circle (0.2cm);
\draw (w) circle (0.2cm);
\begin{scope}[{shift={(a2)},scale=2}]
\draw (-0.15,0.1) -- (0.15,-0.1) -- (0.15,0.1) -- (-0.15,-0.1) -- (-0.15,0.1);
\end{scope}
 \end{tikzpicture}$ & $ab^2d + a^2bc$ \\
\hline
$\begin{tikzpicture}[scale=0.3]
\coordinate (w) at (135:1);
\coordinate (a2) at (45:2) ;
\coordinate (a1) at (135:2) ;
\coordinate (a3) at (315:2) ;
\coordinate (a4) at (225:2) ;
\node at (45:2.6) {$2$};
\node at (135:2.6) {$1$} ;
\node at (315:2.6){$3$} ;
\node at (225:2.6){$4$} ;
\draw (0,0) circle (2);
\draw[->-] (a2) -- (w);
\draw[->-] (a4) -- (w);
\draw[->-] (a1) -- (w);
\begin{scope}[{shift={(a3)},scale=2}]
\draw (-0.15,0.1) -- (0.15,-0.1) -- (0.15,0.1) -- (-0.15,-0.1) -- (-0.15,0.1);
\end{scope}
\filldraw [white] (w) circle (0.2cm);
\draw (w) circle (0.2cm);
 \end{tikzpicture}$ & $ad$ \\
\hline
$\begin{tikzpicture}[scale=0.3]
\coordinate (w) at (45:1);
\coordinate (a2) at (45:2) ;
\coordinate (a1) at (135:2) ;
\coordinate (a3) at (315:2) ;
\coordinate (a4) at (225:2) ;
\node at (45:2.6) {$2$};
\node at (135:2.6) {$1$} ;
\node at (315:2.6){$3$} ;
\node at (225:2.6){$4$} ;
\draw (0,0) circle (2);
\draw[->-] (a1) -- (w);
\draw[->-] (a2) -- (w);
\draw[->-] (a3) -- (w);
\begin{scope}[{shift={(a4)},scale=2}]
\draw (-0.15,0.1) -- (0.15,-0.1) -- (0.15,0.1) -- (-0.15,-0.1) -- (-0.15,0.1);
\end{scope}
\filldraw [white] (w) circle (0.2cm);
\draw (w) circle (0.2cm);
 \end{tikzpicture}$ & $bcd^2+ac^2d$ \\
\hline
\hline
$\begin{tikzpicture}[scale=0.3]
\coordinate (w) at (0:1);
\coordinate (b) at (180:1);
\coordinate (a2) at (45:2) ;
\coordinate (a1) at (135:2) ;
\coordinate (a3) at (315:2) ;
\coordinate (a4) at (225:2) ;
\node at (45:2.6) {$2$};
\node at (135:2.6) {$1$} ;
\node at (315:2.6){$3$} ;
\node at (225:2.6){$4$} ;
\draw (0,0) circle (2);
\draw[->-] (a2) -- (w);
\draw[->-] (a3) -- (w);
\draw[->-] (b) -- (w);
\draw[->-] (b) -- (a1);
\draw[->-] (b) -- (a4);
\filldraw [white] (w) circle (0.2cm);
\draw (w) circle (0.2cm);
\filldraw [black] (b) circle (0.2cm);
 \end{tikzpicture}$ & $bcd$ \\
\hline
$\begin{tikzpicture}[scale=0.3]
\coordinate (w) at (-90:1);
\coordinate (b) at (90:1);
\coordinate (a2) at (45:2) ;
\coordinate (a1) at (135:2) ;
\coordinate (a3) at (315:2) ;
\coordinate (a4) at (225:2) ;
\node at (45:2.6) {$2$};
\node at (135:2.6) {$1$} ;
\node at (315:2.6){$3$} ;
\node at (225:2.6){$4$} ;
\draw (0,0) circle (2);
\draw[->-] (a3) -- (w);
\draw[->-] (a4) -- (w);
\draw[->-] (b) -- (w);
\draw[->-] (b) -- (a1);
\draw[->-] (b) -- (a2);
\filldraw [white] (w) circle (0.2cm);
\draw (w) circle (0.2cm);
\filldraw [black] (b) circle (0.2cm);
 \end{tikzpicture}$ & $abc$ \\
\hline
$\begin{tikzpicture}[scale=0.3]
\coordinate (w) at (180:1);
\coordinate (b) at (0:1);
\coordinate (a2) at (45:2) ;
\coordinate (a1) at (135:2) ;
\coordinate (a3) at (315:2) ;
\coordinate (a4) at (225:2) ;
\node at (45:2.6) {$2$};
\node at (135:2.6) {$1$} ;
\node at (315:2.6){$3$} ;
\node at (225:2.6){$4$} ;
\draw (0,0) circle (2);
\draw[->-] (a1) -- (w);
\draw[->-] (a4) -- (w);
\draw[->-] (b) -- (w);
\draw[->-] (b) -- (a2);
\draw[->-] (b) -- (a3);
\filldraw [white] (w) circle (0.2cm);
\draw (w) circle (0.2cm);
\filldraw [black] (b) circle (0.2cm);
 \end{tikzpicture}$ & $abd$ \\
\hline
$\begin{tikzpicture}[scale=0.3]
\coordinate (w) at (90:1);
\coordinate (b) at (-90:1);
\coordinate (a2) at (45:2) ;
\coordinate (a1) at (135:2) ;
\coordinate (a3) at (315:2) ;
\coordinate (a4) at (225:2) ;
\node at (45:2.6) {$2$};
\node at (135:2.6) {$1$} ;
\node at (315:2.6){$3$} ;
\node at (225:2.6){$4$} ;
\draw (0,0) circle (2);
\draw[->-] (a1) -- (w);
\draw[->-] (a2) -- (w);
\draw[->-] (b) -- (w);
\draw[->-] (b) -- (a3);
\draw[->-] (b) -- (a4);
\filldraw [white] (w) circle (0.2cm);
\draw (w) circle (0.2cm);
\filldraw [black] (b) circle (0.2cm);
 \end{tikzpicture}$ & $acd$ \\
\hline

\end{tabular}
\end{center}
\end{multicols}
Note that all the 50 polynomials are linearly independent.  For example, to obtain the answer for $D = \{2\} \cup (3 \to 4)$ we need to consider the following weblike graphs, contributing $2a^2bcd$, $ab^2d^2$, and $a^3c^2$ respectively.
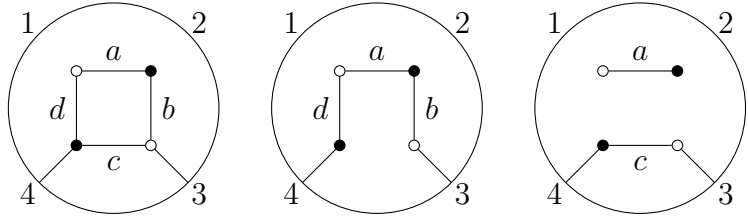
\begin{center}
\begin{tikzpicture}[scale=0.7]
\coordinate (b) at (45:1);
\coordinate (w) at (135:1);
\coordinate (b1) at (225:1);
\coordinate (w1) at (315:1);
\node at (135:2.3) {$1$} ;
\node at (45:2.3) {$2$} ;
\node at (315:2.3){$3$} ;
\node at (225:2.3){$4$} ;
\draw (0,0) circle (2);
\draw (w) -- node[above]{$a$} (b) --node[right]{$b$} (w1) --node[below]{$c$} (b1) --node[left]{$d$} (w);
\draw (w1) --  (315:2);
\draw (b1) --  (225:2);
\filldraw [white] (w) circle (0.1cm);
\filldraw [black] (b) circle (0.1cm);
\draw (w) circle (0.1cm);
\filldraw [white] (w1) circle (0.1cm);
\filldraw [black] (b1) circle (0.1cm);
\draw (w1) circle (0.1cm);

\begin{scope}[{shift={(5,0)}}]
\coordinate (b) at (45:1);
\coordinate (w) at (135:1);
\coordinate (b1) at (225:1);
\coordinate (w1) at (315:1);
\node at (135:2.3) {$1$} ;
\node at (45:2.3) {$2$} ;
\node at (315:2.3){$3$} ;
\node at (225:2.3){$4$} ;
\draw (0,0) circle (2);
\draw (w) -- node[above]{$a$} (b) --node[right]{$b$} (w1);
\draw (b1) --node[left]{$d$} (w);
\draw (w1) --  (315:2);
\draw (b1) --  (225:2);
\filldraw [white] (w) circle (0.1cm);
\filldraw [black] (b) circle (0.1cm);
\draw (w) circle (0.1cm);
\filldraw [white] (w1) circle (0.1cm);
\filldraw [black] (b1) circle (0.1cm);
\draw (w1) circle (0.1cm);
\end{scope}
\begin{scope}[{shift={(10,0)}}]
\coordinate (b) at (45:1);
\coordinate (w) at (135:1);
\coordinate (b1) at (225:1);
\coordinate (w1) at (315:1);
\node at (135:2.3) {$1$} ;
\node at (45:2.3) {$2$} ;
\node at (315:2.3){$3$} ;
\node at (225:2.3){$4$} ;
\draw (0,0) circle (2);
\draw (w) -- node[above]{$a$} (b);
\draw (w1) -- node[below]{$c$} (b1);
\draw (w1) -- (315:2);
\draw (b1) --  (225:2);
\filldraw [white] (w) circle (0.1cm);
\filldraw [black] (b) circle (0.1cm);
\draw (w) circle (0.1cm);
\filldraw [white] (w1) circle (0.1cm);
\filldraw [black] (b1) circle (0.1cm);
\draw (w1) circle (0.1cm);
\end{scope}
\end{tikzpicture}
\end{center}
Note that the leftmost graph has an elliptic web, which we must first reduce.

\end{example}

Write 
$$
\D(W): = \{D \in \D_{k,n} \mid W_D \neq 0\} \qquad \text{and} \qquad \D(G):= \{D \in \D_{k,n} \mid W(G)_D \neq 0\}.$$

\begin{proposition}\label{prop:web}
Suppose $N$ and $N'$ are such that $M(N) = M(N')$.  Then $F_D(N) = \alpha^3 \; F_D(N')$ where the scalar $\alpha$ is given by $\tM(N) = \alpha\tM(N')$.
\end{proposition}
\begin{proof}
By Theorem \ref{thm:Pos}, it suffices to consider the gauge equivalences and local moves (M1-2) and (R1-3).  Suppose $N'$ is obtained from $N$ by multiplying all edge weights incident to a vertex $v$ by $\alpha$.  Then $\tM(N') = \alpha\tM(N)$ and $F_D(N') = \alpha^3F_D(N)$ for any $D$.  The moves (M2) and (R1-3) are similarly easy.

Suppose $N'$ is obtained from $N$ by applying the spider/square move (M1) in Section \ref{sec:moves}.  Due to the confluence of reduction, to find the relationship between $F_D(N')$ and $F_D(N)$ it suffices to compute $F_D$ for $N$ and $N'$ being the two graphs
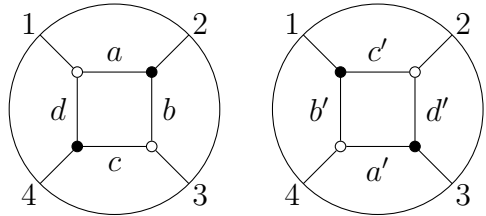
\begin{center}
\begin{tikzpicture}[scale=0.7]
\coordinate (b) at (45:1);
\coordinate (w) at (135:1);
\coordinate (b1) at (225:1);
\coordinate (w1) at (315:1);
\node at (135:2.3) {$1$} ;
\node at (45:2.3) {$2$} ;
\node at (315:2.3){$3$} ;
\node at (225:2.3){$4$} ;
\draw (0,0) circle (2);
\draw (w) -- node[above]{$a$} (b) --node[right]{$b$} (w1) --node[below]{$c$} (b1) --node[left]{$d$} (w);
\draw (w) -- (135:2);
\draw (b) -- (45:2);
\draw (w1) --  (315:2);
\draw (b1) --  (225:2);
\filldraw [white] (w) circle (0.1cm);
\filldraw [black] (b) circle (0.1cm);
\draw (w) circle (0.1cm);
\filldraw [white] (w1) circle (0.1cm);
\filldraw [black] (b1) circle (0.1cm);
\draw (w1) circle (0.1cm);

\begin{scope}[{shift={(5,0)}}]

\coordinate (w) at (45:1);
\coordinate (b) at (135:1);
\coordinate (w1) at (225:1);
\coordinate (b1) at (315:1);
\node at (135:2.3) {$1$} ;
\node at (45:2.3) {$2$} ;
\node at (315:2.3){$3$} ;
\node at (225:2.3){$4$} ;
\draw (0,0) circle (2);
\draw (w) -- node[above]{$c'$} (b) --node[left]{$b'$} (w1) --node[below]{$a'$} (b1) --node[right]{$d'$} (w);
\draw (b) -- (135:2);
\draw (w) -- (45:2);
\draw (b1) --  (315:2);
\draw (w1) --  (225:2);
\filldraw [white] (w) circle (0.1cm);
\filldraw [black] (b) circle (0.1cm);
\draw (w) circle (0.1cm);
\filldraw [white] (w1) circle (0.1cm);
\filldraw [black] (b1) circle (0.1cm);
\draw (w1) circle (0.1cm);
\end{scope}
\end{tikzpicture}
\end{center}
where $a,b,c,d$ and $a',b',c',d'$ are related as in the local move (M1), see Section \ref{sec:moves}.  We compute directly that $\tM(N) = (ac+bd) \tM(N')$.  

To check the statement in theorem, we use the following symmetry.  Suppose $D'$ is obtained from $D$ by 90 degree rotation, sending each boundary vertex $i$ to $i+1 \mod 4$.  Then
$F_{D'}(N')(a',b',c',d') = F_D(N)(d',a',b',c')$.  Thus it suffices to check that for every $D$ we have
$$
F_{D'}(N)(a,b,c,d) = (ac+bd)^3 F_D(N)(d',a',b',c').
$$  
This follows from the tables in Example \ref{ex:big}.
\end{proof}

The non-trivial local move (M1) for planar bipartite graphs involves a square shape, as does the most interesting reduction move for webs.  The calculation in Proposition \ref{prop:web} relates these two moves.  As a corollary, we obtain

\begin{corollary}
For each $D \in \D_{k,n}$, the function $F_D(N)$ depends only on $\tM(N) \in \tGr(k,n)$, and is a degree three element of the homogeneous coordinate ring $\C[\Gr(k,n)]$.
\end{corollary}

\subsection{Restriction to positroid varieties}

Let $N$ be a planar bipartite graph.  Define
$$
\D(N) := \{D \in \D_{k,n} \mid F_D(N) \neq 0\}.
$$
If $\M$ is a positroid of rank $k$ on $[n]$, then we define $\D(\M) = \D(N)$ where $N$ represents $\M$.

\begin{lemma}
$\D(\M)$ does not depend on the choice of $N$.
\end{lemma}
\begin{proof}
Follows immediately from Proposition \ref{prop:web}.
\end{proof}

\begin{remark}
It would be interesting to find the analogue of Oh's theorem (see Remark \ref{rem:TL}), and the analogues of Proposition \ref{prop:TL} and Conjecture \ref{conj:TL} for $\D_{k,n}$ and $\D(\M)$.
\end{remark}

\begin{theorem}\label{thm:web}
The set
$$
\{F_{D} \mid D \in \D(\M)\}
$$
is a basis for the space of functions on $\Pi_\M$ spanned by $\{\Delta_I\Delta_J\Delta_K\}$.  Equivalently, this set forms a basis for the degree three component of the homogeneous coordinate ring $\C[\Pi_\M]$.
\end{theorem}

Let $N$ be obtained from $N'$ by adding a bridge $e$,  black at $i$ to white at $i+1$.  Denote the edges joining $e$ to the boundary vertex $i$ (resp. $i+1$) by $e_i$ (resp. $e_{i+1}$), as illustrated here:
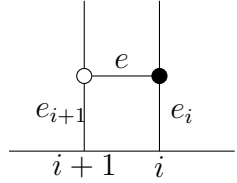
\begin{center}
\begin{tikzpicture}
\draw (1,0) -- (4,0);
\draw (2,0) -- (2,2);
\draw (3,0) -- (3,2) ;
\draw (2,1) -- (3,1);
\filldraw [white] (2,1) circle (0.1cm);
\filldraw [black] (3,1) circle (0.1cm);
\draw (2,1) circle (0.1cm);
\node at (2,-0.2) {$i+1$};
\node at (3,-0.2) {$i$};
\node at (2.5,1.2) {$e$};
\node at (1.7,0.5) {$e_{i+1}$};
\node at (3.3,0.5) {$e_{i}$};
\end{tikzpicture}
\end{center}

The following is straightforward to check.

\begin{lemma}\label{lem:stableadd}
Suppose $D' \in \D(N')$ and $G'$ represents $D'$.  Let $G \subset N$ be weblike such that $G \cap N' = G'$ and $G \setminus G'$ contains $e$. Then 
$$
G \setminus G' = \begin{cases} \{e\} \text{ or } \{e,e_i,e_{i+1}\} &\mbox{$i$ is a source in $D'$ and $i+1$ is a sink in $D'$} \\
\{e\}  &\mbox{$i$ is a sink in $D'$ and $i+1$ is a source in $D'$} \\
\{e,e_i\} & \mbox{$i$ and $i+1$ are sources in $D'$} \\
\{e,e_{i+1}\} & \mbox{$i$ and $i+1$ are sinks in $D'$} \\
\{e, e_i\} & \mbox{$i$ is not used and not tagged, but $i+1$ is a source in $D'$} \\
\{e, e_{i+1}\} & \mbox{$i+1$ is not used and not tagged, but $i$ is a sink in $D'$} \\
\{e, e_{i}\}  \text{ or }  \{e,e_i,e_{i+1}\}& \mbox{$i$ is not used and not tagged, but $i+1$ is a sink in $D'$}\\
\{e, e_{i+1}\} \text{ or }  \{e,e_i,e_{i+1}\}& \mbox{$i+1$ is not used and not tagged, but $i$ is a source in $D'$} \\
\{e\} & \mbox{$i$ and $i+1$ are not used and not tagged in $D'$}
\end{cases}
$$
Furthermore, other $D'$ cannot occur in this way.
\end{lemma}
We illustrate the two possibilities where $i$ is not used and not tagged (so the boundary vertex $i$ in $D'$ is not incident to any edges in $G'$), but $i+1$ is a sink in $D'$.  Here the blue edges are the ones in $G \setminus G'$.
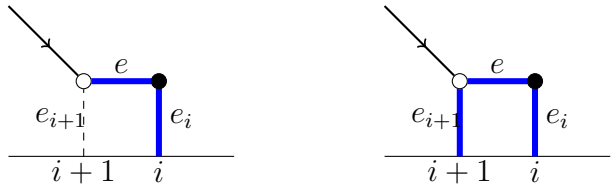
\begin{center}
\begin{tikzpicture}
\coordinate (i) at (3,0);
\coordinate (i1) at (2,0);
\coordinate (b) at (3,1);
\coordinate (w) at (2,1);
\draw[thick,->-] (1,2) -- (w);
\draw (1,0) -- (4,0);
\draw[dashed] (i1) -- (w);
\draw[blue,line width =0.08cm] (i) -- (b) ;
\draw[blue,line width =0.08cm] (w) -- (b);
\filldraw [white] (w) circle (0.1cm);
\filldraw [black] (b) circle (0.1cm);
\draw (w) circle (0.1cm);
\node at (2,-0.2) {$i+1$};
\node at (3,-0.2) {$i$};
\node at (2.5,1.2) {$e$};
\node at (1.7,0.5) {$e_{i+1}$};
\node at (3.3,0.5) {$e_{i}$};
\begin{scope}[shift={(5,0)}]
\coordinate (i) at (3,0);
\coordinate (i1) at (2,0);
\coordinate (b) at (3,1);
\coordinate (w) at (2,1);
\draw[thick,->-] (1,2) -- (w);
\draw (1,0) -- (4,0);
\draw[blue,line width =0.08cm] (i1) -- (w);
\draw[blue,line width =0.08cm] (i) -- (b) ;
\draw[blue,line width =0.08cm] (w) -- (b);
\filldraw [white] (w) circle (0.1cm);
\filldraw [black] (b) circle (0.1cm);
\draw (w) circle (0.1cm);
\node at (2,-0.2) {$i+1$};
\node at (3,-0.2) {$i$};
\node at (2.5,1.2) {$e$};
\node at (1.7,0.5) {$e_{i+1}$};
\node at (3.3,0.5) {$e_{i}$};
\end{scope}
\end{tikzpicture}
\end{center}
Note that the two cases are distinguished by the fact that the edge $e$ contributes $\wt(e)^2$ in the picture on the left, but contributes $\wt(e)^1$ in the picture on the right.

Abusing notation, we say that $D$ is obtained from $D'$ by adding the edges $e, e_i$, and/or $e_{i+1}$ if for a weblike graph $G'$ with $W(G') = D'$, we have $W(G) = D$ where $G$ is obtained from $G'$ by adding the same edges.  Let $D' \in D(N')$ be an irreducible web for $N'$.  We say that $D'$ is \defn{stable} if either $D'$ does not use $i$ or $i+1$, or if joining $i$ and $i+1$ does not cause $D'$ to become non-elliptic.  

\begin{lemma}\label{lem:stable}
Suppose $D \in \D(N) \setminus \D(N')$.  Then there exists a stable $D' \in \D(N')$ such that $D$ is obtained from $D'$ by adding the edge $e$, and some (possibly empty) subset of the edges $\{e_i, e_{i+1}\}$.  
\end{lemma}
\begin{proof}
We use Lemma \ref{lem:webrepresent} in the following.  Suppose $D \in \D(N)$ is represented by a weblike graph $G$.  Then $G' = G \cap N'$ is a weblike subgraph of $N'$.  If $G \setminus G'$ does not include the edge $e$, then we have $W(G') = W(G)$ and so $D \in \D(N')$ as well.  For example, if $e_{i+1}$ is an isolated dipole in $G$ then the boundary vertex $i+1$ is tagged in both $W(G')$ and $W(G)$.  
Thus if $D \in \D(N) \setminus \D(N')$, we must have $e \in G \setminus G'$.
\end{proof}

\begin{proof} [Proof of Theorem \ref{thm:web}]
In Theorem \ref{thm:triple} we will show that $\{F_D \mid D \in \D_{k,n}\}$ span the degree three component of $\C[\Gr(k,n)]$.  Since the restriction map $\C[\Gr(k,n)] \to \C[\Pi_\M]$ is surjective, it remains to show that $\{F_D \mid D \in \D(\M)\}$ is linearly independent in $\C[\Pi_\M]$.

The general strategy of the proof is the same as the proof of Theorem \ref{thm:TLbasis}.  We use the same setup and notation here.  Suppose there exists a linear relation
$$
\sum_{D \in \D(\M)} a_{D} p_{D}(t) = 0.
$$
The same argument as in the proof of Theorem \ref{thm:TLbasis} gives $a_D =0$ if $D \in \D(\M')$. 

Now suppose that $D \in \D(\M)\setminus \D(\M')$.  By Lemma \ref{lem:stable} there exists some stable $D' \in \D(\M')$ such that $D$ is obtained from $D'$ by adding some of the edges $e, e_i, e_{i+1}$.  Whether or not $i$ (resp. $i+1$) are used in $D$ tells us whether or not the edge $e_i$ (resp. $e_{i+1}$) is added.

We deduce that $D'$ is uniquely determined by $D$ except for one situation: when $D$ uses neither $i$ nor $i+1$, in which case $G$ could either (a) have $e$ as an isolated dipole, or (b) have $e$ belong to a component of $G'$ that uses both $i$ and $i+1$.

Let $V' = {\rm span}(F_{D'}(N') \mid D' \text{ is unstable})$.  If the stable $D'$ is uniquely determined, then we have 
\begin{equation}\label{eq:D1}
p_D(t) = t^a F_{D'} \mod V'\otimes \C[t]
\end{equation}
where $a \in \{1,2,3\}$.  In the case that the stable $D'$ is not uniquely determined we have
\begin{equation}\label{eq:D2}
p_D(t) = t^3 F_{D'} + t^a F_{D''} \mod V'\otimes \C[t]
\end{equation}
where $a \in \{1,2\}$, for stable $D', D'' \in \D(N')$.

By Lemma \ref{lem:stableadd}, each stable $D'$ either occurs once in \eqref{eq:D1} or \eqref{eq:D2}, or it occurs twice, but with different powers of $t$.  It follows that for a fixed $a \in \{1,2,3\}$, the coefficient of $F_{D'}(N')$ in $[t^a]\sum_{D \in 
\D(\M)} a_{D} p_{D}(t)$ is either $0$ or a single $a_D$.  This proves that $a_D = 0$ for all $D \in  \D(\M)\setminus \D(\M')$, as required.
\end{proof}

By standard results about the homogeneous coordinate ring $\C[\Gr(k,n)]$, we have
\begin{corollary}
The cardinality of $\D_{k,n}$ is equal to the number of semistandard tableaux of shape $3^k$, filled with numbers $\{1,2,\ldots,n\}$.
\end{corollary}
The reader is invited to check that this agrees with $|\D_{2,4}| = 50$.

\subsection{Products of three minors}
Many ideas in this section are already present in Pylyavskyy \cite{Pyl}.  Pylyavskyy works in the setting of a $n \times n$ matrix.  We have modified the results for the Grassmannian situation, and we believe also simplified the presentation.

Let $W$ be a web on $[n]$.   A \defn{labeling} $(W,\alpha)$ of $W$ is an assignment of one of the three labels $\{1,2,3\}$ to each non-isolated edge in $W$ with the property that the three edges incident to any trivalent vertex have distinct labels.  Internal loops in $W$ are labeled by a single label.  (One can also think of isolated edges as labeled by all three labels.)

\begin{lemma}
Every web $W$ has a labeling.
\end{lemma}
\begin{proof}
This can be proved by induction on the number of vertices of the web $W$.  Pick two adjacent boundary vertices of some component of $W$, and find a self-avoiding path between these two vertices.  Then we get a situation that looks like:
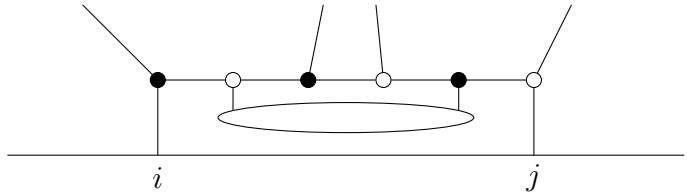
\begin{center}
\begin{tikzpicture}
\node at (0,-0.3) {$i$};
\node at (5,-0.3) {$j$};
\draw (-2,0) -- (7,0);
\draw (0,1) -- (5,1);
\draw (0,0) -- (0,1);
\draw (5,0) -- (5,1);
\draw (0,1) -- (-1,2); 
\draw (1,1) -- (1,0.5); 
\draw (2,1) -- (2.2,2);
\draw (3,1) -- (2.9,2);
\draw (4,1) -- (4,0.5); 
\draw (5,1) -- (5.5,2);
\filldraw[white] (2.5,0.5) ellipse (1.7cm and 0.2cm);
\draw (2.5,0.5) ellipse (1.7cm and 0.2cm);

\filldraw[black] (0,1) circle (0.1cm);
\filldraw[black] (2,1) circle (0.1cm);
\filldraw[black] (4,1) circle (0.1cm);
\filldraw[white] (1,1) circle (0.1cm);
\draw (1,1) circle (0.1cm);
\filldraw[white] (3,1) circle (0.1cm);
\draw (3,1) circle (0.1cm);
\filldraw[white] (5,1) circle (0.1cm);
\draw (5,1) circle (0.1cm);
\end{tikzpicture}
\end{center}
By induction, we can first label all the webs hanging off this path.  Then it is easy to see that there is a labeling of the edges in this path.
\end{proof}

Let $(I,J,K)$ be a triple of $k$-element subsets of $[n]$ and denote by $\bar I = [n] \setminus I$ (resp. $\bar J$, resp. $\bar K$) the complement subset.  We say that a labeled web $(W,\alpha)$ is \defn{consistently labeled} with $(I,J,K)$ if for any boundary vertex $i \in [n]$, 
\begin{enumerate}
\item
if $i$ is a black sink or white source in $W$, and the edge incident to $i$ is labeled by a $1$ (resp. $2$, resp. $3$), then $i \in \bar I \cap J \cap K$  (resp. $I \cap \bar J \cap K$, resp. $I \cap J \cap \bar K$);
\item
if $i$ is a white sink or black source in $W$, and the edge incident to $i$ is labeled by a $1$ (resp. $2$, resp. $3$), then $i \in I\cap \bar J \cap \bar K$  (resp. $\bar I\cap J \cap \bar K$, resp. $\bar I\cap \bar J \cap K$);
\item
if $i$ is black and tagged or white and untagged in $W$ then $i \in I \cap J \cap K$;
\item
if $i$ is white and tagged or black and untagged in $W$ then $i \in \bar I \cap \bar J \cap \bar K$.
\end{enumerate}
Let $a(I,J,K;W)$ denote the number of consistent labelings of $W$ with $(I,J,K)$.

\begin{lemma}\label{lem:pyl}
Suppose $W = \sum_{D \in \D_{k,n}} W_D D$ is the decomposition of $W$ into non-elliptic webs.  Then
$$
a(I,J,K;W) = \sum_{D \in \D_{k,n}} W_D \; a(I,J,K;D).
$$
\end{lemma}
\begin{proof}
The identity is checked case-by-case for each of the elementary reduction moves of Section \ref{sec:moves}.  For example, if we apply move (M1), we have
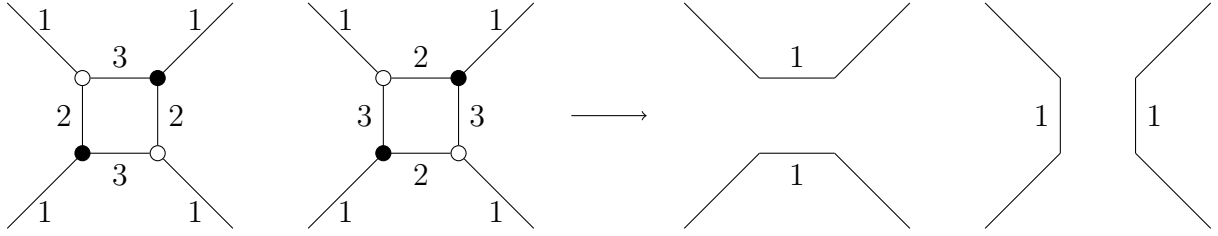
\begin{center}
\begin{tikzpicture}
\coordinate (w1) at (3,0);
\coordinate (b1) at (2,0);
\coordinate (b) at (3,1);
\coordinate (w) at (2,1);
\draw (w) -- node[above] {$3$} (b) -- node[right] {$2$} (w1) -- node[below] {$3$} (b1) -- node[left] {$2$}(w);
\draw (w) -- node[above] {$1$} (1,2);
\draw (b) -- node[above] {$1$}(4,2);
\draw (w1) -- node[below] {$1$} (4,-1);
\draw (b1) -- node[below] {$1$} (1,-1);
\filldraw [white] (w) circle (0.1cm);
\filldraw [black] (b) circle (0.1cm);
\draw (w) circle (0.1cm);
\filldraw [white] (w1) circle (0.1cm);
\filldraw [black] (b1) circle (0.1cm);
\draw (w1) circle (0.1cm);
\begin{scope}[shift={(4,0)}]
\coordinate (w1) at (3,0);
\coordinate (b1) at (2,0);
\coordinate (b) at (3,1);
\coordinate (w) at (2,1);
\draw (w) -- node[above] {$2$} (b) -- node[right] {$3$} (w1) -- node[below] {$2$} (b1) -- node[left] {$3$}(w);
\draw (w) -- node[above] {$1$} (1,2);
\draw (b) -- node[above] {$1$}(4,2);
\draw (w1) -- node[below] {$1$} (4,-1);
\draw (b1) -- node[below] {$1$} (1,-1);
\filldraw [white] (w) circle (0.1cm);
\filldraw [black] (b) circle (0.1cm);
\draw (w) circle (0.1cm);
\filldraw [white] (w1) circle (0.1cm);
\filldraw [black] (b1) circle (0.1cm);
\draw (w1) circle (0.1cm);

\end{scope}
\draw[->] (8.5,0.5) -- (9.5,0.5);
\begin{scope}[shift={(9,0)}]
\coordinate (w1) at (3,0);
\coordinate (b1) at (2,0);
\coordinate (b) at (3,1);
\coordinate (w) at (2,1);
\draw (w) -- node[above] {$1$} (b) ;
\draw (w1) -- node[below] {$1$} (b1);
\draw (w) -- (1,2);
\draw (b) -- (4,2);
\draw (w1) --  (4,-1);
\draw (b1) -- (1,-1);
\end{scope}
\begin{scope}[shift={(13,0)}]
\coordinate (w1) at (3,0);
\coordinate (b1) at (2,0);
\coordinate (b) at (3,1);
\coordinate (w) at (2,1);
\draw (b) -- node[right] {$1$} (w1);
\draw (b1) -- node[left] {$1$}(w);
\draw (w) -- (1,2);
\draw (b) -- (4,2);
\draw (w1) --  (4,-1);
\draw (b1) -- (1,-1);
\end{scope}
\end{tikzpicture}
\end{center}
when the boundary edges of the square all have the same label, and
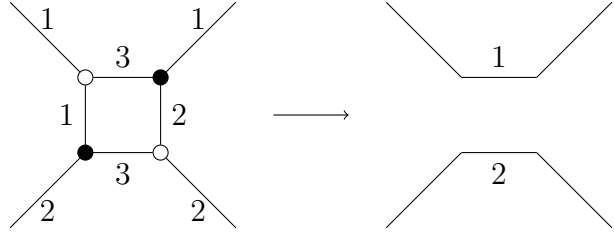
\begin{center}
\begin{tikzpicture}
\coordinate (w1) at (3,0);
\coordinate (b1) at (2,0);
\coordinate (b) at (3,1);
\coordinate (w) at (2,1);
\draw (w) -- node[above] {$3$} (b) -- node[right] {$2$} (w1) -- node[below] {$3$} (b1) -- node[left] {$1$}(w);
\draw (w) -- node[above] {$1$} (1,2);
\draw (b) -- node[above] {$1$}(4,2);
\draw (w1) -- node[below] {$2$} (4,-1);
\draw (b1) -- node[below] {$2$} (1,-1);
\filldraw [white] (w) circle (0.1cm);
\filldraw [black] (b) circle (0.1cm);
\draw (w) circle (0.1cm);
\filldraw [white] (w1) circle (0.1cm);
\filldraw [black] (b1) circle (0.1cm);
\draw (w1) circle (0.1cm);

\draw[->] (4.5,0.5) -- (5.5,0.5);
\begin{scope}[shift={(5,0)}]
\coordinate (w1) at (3,0);
\coordinate (b1) at (2,0);
\coordinate (b) at (3,1);
\coordinate (w) at (2,1);
\draw (w) -- node[above] {$1$} (b) ;
\draw (w1) -- node[below] {$2$} (b1);
\draw (w) -- (1,2);
\draw (b) -- (4,2);
\draw (w1) --  (4,-1);
\draw (b1) -- (1,-1);
\end{scope}
\end{tikzpicture}
\end{center}
when they do not.
\end{proof}

\begin{theorem}\label{thm:triple}
As functions on the cone over the Grassmannian, we have
$$
\Delta_I \Delta_J \Delta_K = \sum_{D \in \D_{k,n}} a(I,J,K;D) F_D.
$$
\end{theorem}
\begin{proof}
Let $N$ be a planar bipartite graph with boundary vertices $[n]$.  We have
$$
\Delta_I(N) \Delta_J(N) \Delta_K(N) = \sum_{\Pi_1,\Pi_2,\Pi_3} \wt(\Pi_1) \wt(\Pi_2) \wt(\Pi_3)
$$
where the summation is over triples $(\Pi_1,\Pi_2,\Pi_3)$ of dimer configurations with boundary configurations $I(\Pi_1) = I$, $I(\Pi_2) = J$, and $I(\Pi_3) = K$ respectively.  Overlaying these dimer configurations on top of each other, we obtain a weblike subgraph $G \subset N$ and a labeling $\alpha$ of $W(G)$: isolated dipoles are edges occurring in all three dimer configurations, and for a path whose endpoints are trivalent (but other vertices are bivalent), we do the following to obtain $\alpha$:
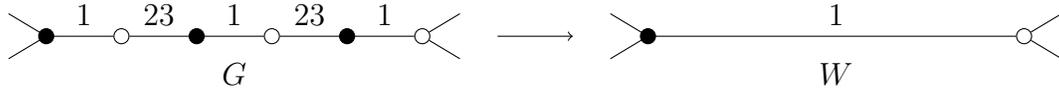
\begin{center}
\begin{tikzpicture}
\coordinate (b) at (0,0);
\coordinate (w) at (1,0);
\coordinate (b2) at (2,0);
\coordinate (w2) at (3,0);
\coordinate (b3) at (4,0);
\coordinate (w3) at (5,0);
\draw (-0.5,0.3) -- (b) -- (-0.5 ,-0.3);
\draw (5.5,0.3) -- (w3) -- (5.5 ,-0.3);
\draw (b) -- node[above] {$1$} (w) -- node[above]{$23$} (b2)  -- node[above] {$1$} (w2) -- node[above]{$23$} (b3)  -- node[above] {$1$} (w3);
\filldraw [white] (w) circle (0.1cm);
\filldraw [black] (b) circle (0.1cm);
\draw (w) circle (0.1cm);
\filldraw [black] (b2) circle (0.1cm);
\filldraw [white] (w2) circle (0.1cm);
\draw (w2) circle (0.1cm);
\filldraw [black] (b3) circle (0.1cm);
\filldraw [white] (w3) circle (0.1cm);
\draw (w3) circle (0.1cm);
\draw [->] (6,0) -- (7,0);
\node at (2.5,-0.5) {$G$};
\begin{scope}[shift={(8,0)}]
\coordinate (b) at (0,0);
\coordinate (w) at (1,0);
\coordinate (b2) at (2,0);
\coordinate (w2) at (3,0);
\coordinate (b3) at (4,0);
\coordinate (w3) at (5,0);
\draw (-0.5,0.3) -- (b) -- (-0.5 ,-0.3);
\draw (5.5,0.3) -- (w3) -- (5.5 ,-0.3);
\node at (2.5,-0.5) {$W$};
\draw (b) -- node[above] {$1$} (w3);
\filldraw [black] (b) circle (0.1cm);
\filldraw [white] (w3) circle (0.1cm);
\draw (w3) circle (0.1cm);
\end{scope}
\end{tikzpicture}
\end{center}
Here an edge labeled by $1$ in $G$ indicates an edge that is present in only $\Pi_1$, while the edges labeled $23$ are present in both $\Pi_2$ and $\Pi_3$ but not $\Pi_1$.  Conversely, a weblike subgraph $G \subset N$ together with a consistent labeling $(W,\alpha)$ with $(I,J,K)$ arises from a triple of dimer configurations.  Comparing with the definition of weight of $G$ we have
$$
\Delta_I(N) \Delta_J(N) \Delta_K(N) = \sum_W a(I,J,K;W) \sum_{G: W(G) = W} \wt(G).
$$
Note that if a cycle $C_i$ of even length in $G$ is labeled by $I$ in $W$, it comes from two different triples of dimer configurations.  Using Lemma \ref{lem:pyl}, we have
$$
\Delta_I(N) \Delta_J(N) \Delta_K(N) = \sum_D a(I,J,K;D) F_D(N).
$$
Now both sides depend only on the point $M(N) \in \Gr(k,n)$, so we have an identity on the Grassmannian.
\end{proof}

For example, the three dimer configurations
\begin{center}
\begin{tikzpicture}
\coordinate (b) at (45:1);
\coordinate (w) at (135:1);
\coordinate (b1) at (225:1);
\coordinate (w1) at (315:1);
\draw (0,0) circle (2);
\draw (w) -- (b) -- (w1) -- (b1) -- (w);
\draw (w) --  (135:2);
\draw (b) -- (45:2);
\draw[very thick, blue] (w1) --  (315:2);
\draw[very thick, blue] (b1) --  (225:2);
\draw[very thick, blue] (w) -- (b);
\filldraw [white] (w) circle (0.1cm);
\filldraw [black] (b) circle (0.1cm);
\draw (w) circle (0.1cm);
\filldraw [white] (w1) circle (0.1cm);
\filldraw [black] (b1) circle (0.1cm);
\draw (w1) circle (0.1cm);
\begin{scope}[shift={(5,0)}]
\coordinate (b) at (45:1);
\coordinate (w) at (135:1);
\coordinate (b1) at (225:1);
\coordinate (w1) at (315:1);
\draw (0,0) circle (2);
\draw (w) -- (b) -- (w1) -- (b1) -- (w);
\draw[very thick, blue] (w) --  (135:2);
\draw[very thick, blue] (b) -- (45:2);
\draw (w1) --  (315:2);
\draw (b1) --  (225:2);
\draw[very thick, blue] (w1) -- (b1);
\filldraw [white] (w) circle (0.1cm);
\filldraw [black] (b) circle (0.1cm);
\draw (w) circle (0.1cm);
\filldraw [white] (w1) circle (0.1cm);
\filldraw [black] (b1) circle (0.1cm);
\draw (w1) circle (0.1cm);
\end{scope}

\begin{scope}[shift={(10,0)}]
\coordinate (b) at (45:1);
\coordinate (w) at (135:1);
\coordinate (b1) at (225:1);
\coordinate (w1) at (315:1);
\draw (0,0) circle (2);
\draw (w) -- (b) -- (w1) -- (b1) -- (w);
\draw (w) --  (135:2);
\draw (b) -- (45:2);
\draw (w1) --  (315:2);
\draw (b1) --  (225:2);
\draw[very thick, blue] (w) -- (b1);
\draw[very thick, blue] (w1) -- (b);
\filldraw [white] (w) circle (0.1cm);
\filldraw [black] (b) circle (0.1cm);
\draw (w) circle (0.1cm);
\filldraw [white] (w1) circle (0.1cm);
\filldraw [black] (b1) circle (0.1cm);
\draw (w1) circle (0.1cm);
\end{scope}
\end{tikzpicture}
\end{center}
gives the labeled (elliptic) web
\begin{center}
\begin{tikzpicture}
\coordinate (b) at (45:1);
\coordinate (w) at (135:1);
\coordinate (b1) at (225:1);
\coordinate (w1) at (315:1);
\draw (0,0) circle (2);
\draw (w) -- node[above]{$1$} (b) --node[right]{$3$} (w1) --node[below]{$2$} (b1) --node[left]{$3$} (w);
\draw (w) --node[above]{$2$}  (135:2);
\draw (b) --node[above]{$2$} (45:2);
\draw (w1) --node[below]{$1$}  (315:2);
\draw (b1) -- node[below]{$1$} (225:2);
\filldraw [white] (w) circle (0.1cm);
\filldraw [black] (b) circle (0.1cm);
\draw (w) circle (0.1cm);
\filldraw [white] (w1) circle (0.1cm);
\filldraw [black] (b1) circle (0.1cm);
\draw (w1) circle (0.1cm);
\end{tikzpicture}
\end{center}

\end{document}